\documentclass{amsart}
\usepackage{amssymb,amsmath,amsfonts,amscd,euscript}

\newcommand{\comment}[1]{}

\usepackage{hyperref}

\newcommand{\al}{\alpha}
\newcommand{\be}{\beta}
\newcommand{\ga}{\gamma}
\newcommand{\de}{\delta}
\newcommand{\ep}{\epsilon}
\newcommand{\la}{\lambda}
\newcommand{\om}{\omega}
\newcommand{\si}{\sigma}

\newcommand{\bbe}{\overline{\beta}}

\newcommand{\De}{\Delta}
\newcommand{\Ga}{\Gamma}
\newcommand{\La}{\Lambda}
\newcommand{\Om}{\Omega}

\renewcommand{\bar}{\overline}
\renewcommand{\tilde}{\widetilde}
\renewcommand{\hat}{\widehat}

\newcommand{\Z}{\mathbb{Z}}
\newcommand{\Q}{\mathbb{Q}}

\newcommand{\C}{\mathbb{C}}
\newcommand{\K}{\mathbb{K}}
\newcommand{\LL}{\mathbb{L}}
\newcommand{\W}{\mathbb{W}}
\newcommand{\HH}{\mathbb{H}}

\renewcommand{\H}{\mathcal{H}}

\newcommand{\cF}{\mathcal{F}}
\newcommand{\cT}{\mathcal{T}}
\newcommand{\hcF}{\hat{\cF}}
\newcommand{\hcT}{\hat{\cT}}

\newcommand{\B}{\mathcal{B}}
\newcommand{\QB}{\mathcal{QB}}
\newcommand{\QBR}{\overleftarrow{\mathcal{QB}}}
\newcommand{\QBG}{\mathrm{QBG}}

\newcommand{\fg}{\mathfrak{g}}

\newcommand{\fh}{\mathfrak{h}}
\newcommand{\fsl}{\mathfrak{sl}}

\newcommand{\tW}{\tilde{W}}

\DeclareMathOperator{\dir}{\mathrm{dir}}
\DeclareMathOperator{\wt}{\mathrm{wt}}
\DeclareMathOperator{\qwt}{\mathrm{qwt}}
\DeclareMathOperator{\ord}{\mathrm{ord}}
\DeclareMathOperator{\ept}{\mathrm{end}}

\newcommand{\Hom}{\mathrm{Hom}}
\newcommand{\End}{\mathrm{End}}
\newcommand{\id}{\mathrm{id}}
\newcommand{\Inv}{\mathrm{Inv}}
\newcommand{\Fun}{\mathrm{Fun}}
\newcommand{\Frac}{\mathrm{Frac}}

\newcommand{\ut}{(\mathbf{u})}
\newcommand{\du}{(\mathbf{d})}

\newcommand{\twovec}[2]{\left[\begin{array}{c}#1\\#2\end{array}\right]}
\newcommand{\pair}[2]{\langle #1, #2\rangle}

\DeclareMathOperator{\ch}{ch}

\newcommand{\tX}{\tilde{X}}
\newcommand{\tY}{\tilde{Y}}
\newcommand{\tI}{\tilde{I}}
\newcommand{\tS}{\tilde{S}}

\newcommand{\hY}{\hat{Y}}

\newcommand{\uw}{\underline{w}}

\newtheorem{lem}{Lemma}
\newtheorem{thm}[lem]{Theorem}
\newtheorem{prop}[lem]{Proposition}
\newtheorem{cor}[lem]{Corollary}

\theoremstyle{remark}
\newtheorem{rem}[lem]{Remark}

\numberwithin{equation}{section}
\numberwithin{lem}{section}

\newcounter{tmp}

\newcommand{\h}{\mathfrak{h}}

\newcommand{\T}{\otimes}

\newcommand{\msl}{\mathfrak{sl}}

\newcommand{\U}{\mathrm U}

\newcommand{\bC}{{\mathbb C}}

\newcommand{\bZ}{{\mathbb Z}}

\newcommand{\fn}{{\mathfrak n}}

\begin{document}

\title[]{Generalized Weyl modules and nonsymmetric $q$-Whittaker functions}

\author{Evgeny Feigin}
\address{Evgeny Feigin:\newline
Department of Mathematics,\newline
National Research University Higher School of Economics,\newline
Vavilova str. 7, 117312, Moscow, Russia,\newline
{\it and }\newline
Tamm Theory Division, Lebedev Physics Institute
}
\email{evgfeig@gmail.com}

\author{Ievgen Makedonskyi}
\address{Ievgen Makedonskyi:\newline
Department of Mathematics,\newline
National Research University Higher School of Economics,\newline
Vavilova str. 7, 117312, Moscow, Russia
}
\email{makedonskii\_e@mail.ru}

\author{Daniel Orr}
\address{Daniel Orr:\newline
Department of Mathematics (MC 0123)\newline
460 McBryde Hall, Virginia Tech\newline
225 Stanger St.\newline
Blacksburg, VA 24061 USA}
\email{dorr@vt.edu}

\date{\today}

\begin{abstract}
We introduce generalized global Weyl modules and relate their graded characters to nonsymmetric Macdonald polynomials and nonsymmetric $q$-Whittaker functions. In particular, we show that the series part of the nonsymmetric $q$-Whittaker function is a generating function for the graded characters of generalized global Weyl modules.
\end{abstract}

\maketitle

\tableofcontents

\section{Introduction}

\begingroup
\renewcommand\thelem{\Alph{lem}}

Let $\fg$ be a finite-dimensional simple Lie algebra over $\C$. Let $X$ be the weight lattice and $W$ the Weyl group of $\fg$.
Given a Cartan decomposition $\fg=\fn_-\oplus\fh\oplus\fn_+$, let
\begin{align*}
\fn^{af} = \fg\otimes t\C[t] \oplus \fn_+\otimes 1.
\end{align*}
be the affine nilpotent subalgebra of the current algebra $\fg\otimes\C[t]$.
Generalized local Weyl modules $W_{\sigma(\la)}$ were introduced in \cite{FM3}; these are finite-dimensional
cyclic $\fn^{af}$-modules indexed by an antidominant weight $\la\in X_-$ and a Weyl group element $\sigma\in W$.
The dimension of $W_{\sigma(\la)}$ does not depend on $\sigma$, but the $\fn^{af}$-module structure of course does.
The module $W_\la$ (where $\si$ stabilizes $\la$) is a module over $\fg\otimes\C[t]$ isomorphic to the local Weyl module $W(w_0\la)$ of \cite{CP,CL,FL2}, where $w_0$ is the longest element in the Weyl group.

The main goal of the present paper is to introduce the generalized {\em global} Weyl modules and to study their connection to the theory of nonsymmetric Macdonald polynomials and nonsymmetric $q$-Whittaker functions. The global Weyl modules $\W_{\si(\la)}$ are defined as cyclic $\fn^{af}$-modules with a cyclic vector $v$ subject to the following defining relations:
\begin{align*}
(f_{\alpha}\otimes t) v=0,&\quad \forall\,\alpha \in \sigma(\Delta_-)\cap \Delta_-;\\
(e_\alpha\otimes 1) v=0,&\quad \forall\,\alpha \in \sigma(\Delta_-)\cap\Delta_+;\\
(f_{\sigma(\alpha)}\otimes t)^{-\langle \alpha^\vee, \lambda \rangle+1} v=0,&\quad \forall\,\alpha \in \Delta_+\cap \si^{-1}(\Delta_-);\\
(e_{\sigma(\alpha)}\otimes 1)^{-\langle \alpha^\vee, \lambda \rangle+1} v=0,&\quad \forall\,\alpha \in \Delta_+\cap \si^{-1}(\Delta_+).
\end{align*}
Here $\De=\De_+\cup\De_-\subset\fh^*$ are the roots of $\fg$, split into positive and negative subsets according to the Cartan decomposition.
Let $A_{\sigma(\la)}={\rm U}(\fh\T t\bC[t])v$ be the highest weight algebra of $\W_{\sigma(\la)}$.
Our main results on the modules $\W_{\si(\la)}$ can be summarized as follows:
\begin{thm}
\label{T:A}
Let $\la\in X_-$ and $\si\in W$.
\begin{enumerate}
\item The algebra $A_{\sigma(\la)}$ acts freely on $\W_{\sigma(\la)}$, with the quotient isomorphic to $W_{\sigma(\la)}$. (Theorem~\ref{T:struct}(2) and Corollary~\ref{estimate})
\item The algebra $A_{\sigma(\la)}$ does not depend on $\sigma$. In particular, $A_{\sigma(\la)}\simeq A(w_0\la)$, the classical highest weight algebra. (Proposition~\ref{Asi})
\item Suppose that $\langle\la,\al_i^\vee\rangle<0$ for some $i=1,\dots,n$. Then there exists an embedding of $\fn^{af}$-modules
$\W_{\sigma(\la+\om_i)}\to \W_{\sigma(\la)}$. The cokernel can be filtered in such a way that each subquotient is isomorphic to a generalized global
Weyl module of the form $\W_{\tau(\la+\om_i)}$ for some $\tau\in W$. The number of these subquotients is equal to the dimension of the fundamental
local Weyl module $W(\om_i)$. (Theorem~\ref{filtration})
\end{enumerate}
\end{thm}

The main tools we use to prove Theorem~\ref{T:A} are the combinatorics of the quantum Bruhat graph and known properties of generalized local Weyl modules from~\cite{FM3}. We note that modules similar to the generalized Weyl modules were studied in the quantum setting in \cite{INS,NS,NNS}.

In order to establish a connection between generalized Weyl modules and nonsymmetric $q$-Whittaker functions, we study a family of polynomials
$E_\la^\sigma(X;q,v)$ indexed by $\la\in X$ and $\sigma\in W$. Up to a normalization, the polynomial $E_\la^\sigma$ is obtained by applying Demazure-Lusztig operators to the nonsymmetric Macdonald polynomial $E_\la$ (see \eqref{E:def-E-u}). In particular, we have $E_\la^\id=E_\la$. It is immediate from the definition that for any fixed $\sigma\in W$, the set $\{E_\la^\sigma\}_{\la \in X}$ is a basis for the polynomial module of DAHA. The Ram-Yip formula \cite{RY} gives an explicit expression for $E_\la^\si$ in terms of alcove walks. We use this to express the graded characters of local and global Weyl modules (extended to modules over $\fn^{af}\oplus\fh\otimes 1$) in terms of $E_\la^\si$ for $\la\in X_-$ and $\si\in W$; see \eqref{E:mac-weyl} and \eqref{E:mac-glob}. This may be viewed as an extension of results of \cite{CI,I,San} expressing the $v=0$ specializations of Macdonald polynomials as graded characters. We also mention the related works \cite{NNS,NS} and \cite{LNSSS1}--\cite{LNSSS4} in this direction. The works \cite{I,LNSSS4,San} realize the nonsymmetric Macdonald polynomials $E_\la(X;q,0)$ for arbitrary $\la\in X$ as graded characters of Demazure submodules of Weyl modules, while \cite{NNS} considers the $v=\infty$ specializations for $\la\in X$ (which are essentially our $E_\la^{w_0}(X;q,0)$---see \eqref{E:Tw0}).

In type $A$, the polynomials $E_\la^{\sigma}$ have been studied recently by Alexandersson~\cite{A}, using the combinatorics of non-attacking fillings with  general basement (determined by $\si$), which results in a natural extension to $E_\la^\si$ of the Haglund-Haiman-Loehr formula for $E_\la$ \cite{HHL}.

\comment{
There is an auxiliary lattice $Y$, which is either the weight lattice of $\fg$ or that of its Langlands dual $\fg^\vee$. Consider the group algebra $\K[Y]$ with its natural $W$-action. The $q$-Toda difference operators make up a homomorphism of algebras
\begin{align}\label{E:q-toda}
\cT: \K[Y]^W \to \End(\Q(q)[Z]).
\end{align}
For instance, when $\fg=\fsl_2$ and $Y$ is the weight lattice of $\fg$. We have
$$ \cT(Y^\om+Y^{-\om}) = (1-Z^{-2})\Gamma + \Gamma^{-1}. $$
Let $\Fun(W,\Q(q)[Z])$ be the set of all functions $f: W\to\Q(q)[Z]$. We have $\Fun(W,\Q(q)[Z])=\oplus_{\si\in W}\Q(q)[Z]e_\si$ where $e_\si$ is the characteristic function of $\si\in W$. The $q$-Toda Dunkl operators give a lifting of \eqref{E:q-toda} to
\begin{align*}
\hcT: \K[Y] \to \End(\Q(q)[Z]^{|W|})
\end{align*}
}

We apply our study of the polynomials $E_\la^\si$ to establish a connection between the characters of generalized global Weyl modules and the nonsymmetric $q$-Whittaker function of \cite{CO}. The nonsymmetric $q$-Whittaker function $\Om(Z,X)$ is a joint eigenfunction for the $q$-Toda Dunkl operators. These operators are discussed briefly in \S\ref{SS:whitt}, but we remark here that the $q$-Toda Dunkl operators can be symmetrized to the $q$-Toda difference operators studied in \cite{E,R,S,BF1,BF2,FFJMM,GLO,GiL} and elsewhere. The applications of $q$-Toda difference operators are remarkably diverse, and the natural expectation is that the same should be true for the $q$-Toda Dunkl operators. We refer the interested reader to \cite{CO} for more details on $q$-Toda Dunkl operators.

The connection between nonsymmetric $q$-Whittaker functions and generalized Weyl modules can be formulated as follows:

\begin{thm}[Theorem~\ref{T:whitt}]
\label{T:B}
Suppose that $\fg$ is simply-laced, and let $(\cdot,\cdot)$ be the invariant form on $\fh^*$ such that $(\al,\al)=2$ for all roots $\al\in\De$. Let $\{e_\si\}_{\si \in W}$ be a basis for a $|W|$-dimensional vector space. Then the vector-valued generating function
\begin{align}\label{E:GF}
\sum_{\la\in X_-} q^{\frac{(\la,\la)}{2}}Z^{-\la} \sum_{\si\in W} \ch \W_{\si(\la)}\, e_\si,
\end{align}
where $q$ is a formal parameter and $\ch \W_{\si(\la)}$ is the graded $q$-character, is the series part of the nonsymmetric $q$-Whittaker function $\Om(Z,X)$.
\end{thm}

One recovers the $q$-Whittaker function $\mathcal{W}$ of \cite{Ch-whitt} by taking the identity component (i.e., the coefficient of $e_\id$) of the nonsymmetric $q$-Whittaker function $\Om(Z,X)$. The function $\mathcal{W}$ is directly related to the $q$-Whittaker functions studied in \cite{BF1,BF2,FFJMM,GLO}.

One consequence of Theorem~\ref{T:B} is that the generating function \eqref{E:GF} satisfies eigenfunction equations with respect to the $q$-Toda Dunkl operators. This boils down to certain recurrence relations for the graded characters $\ch \W_{\si(\la)}$. These recurrence relations are, in turn, a shadow of the decomposition procedure in Theorem~\ref{T:A}(3). See \S\ref{SS:A1-whitt} for an explicit demonstration of this in the $\fsl_2$-case.

We expect that the full DAHA-symmetry of the nonsymmetric $q$-Whittaker function from \cite[Theorem~3.4]{CO} will provide further relations among the $\ch \W_{\si(\la)}$. We plan to investigate this in a future work.

We are able to prove a variant of Theorem~\ref{T:B} in the greater generality of dual untwisted affine root systems (using the right-hand side of \eqref{E:mac-glob} as a substitute for $\ch \W_{\si(\la)}$). While the connection between $q$-Whittaker functions and generalized Weyl modules holds at present only in the simply-laced case, it is expected to hold more generally once the theory of generalized Weyl modules is developed in the twisted cases (see \cite{FM4}). Finally, one can also formulate a version of Theorem~\ref{T:B} in the non-simply-laced untwisted cases, but our method of proof does not apply in this situation.

Theorem~\ref{T:B} gives the nonsymmetric $q$-Whittaker function a proper home in representation theory, beyond the purely DAHA-related considerations of \cite{CO}. Recently S.~Kato \cite{Kato} has shown that the approach of \cite{BF1,BF2} can be extended to the generalized Weyl modules, thus giving a geometric interpretation of some of our results.

The paper is organized as follows. In \S\ref{S:mac} we introduce the combinatorial constructs of the paper, including nonsymmetric Macdonald polynomials, alcove walks, and the Ram-Yip formula. Weyl modules are the focus of \S\ref{S:weyl}, where we first recall known results on classical and generalized Weyl modules and then introduce and study the generalized global Weyl modules. The connection to nonsymmetric $q$-Whittaker functions is established in \S\ref{S:whitt}. Finally, \S\ref{S:A1} gives the main constructions of the paper in detail for $\fg=\fsl_2$. The reader may find it helpful to refer to \S\ref{S:A1} while reading the rest of the paper.

\endgroup
\setcounter{lem}{\thetmp}

\section*{Acknowledgements}
We thank Michael Finkelberg, Mark Shimozono, Ivan Cherednik, and Syu Kato for stimulating discussions. This work began at Mini-Workshop 1609a ``PBW Structures in Representation Theory'' at the Mathematisches Forschungsinstitut Oberwolfach. We thank the organizers of this workshop and the MFO for the hospitality. D.O. gratefully acknowledges support from the National Science Foundation, via DMS-1049268 ``US Junior Oberwolfach Fellows.'' This work has also been supported by the Russian Academic Excellence Project `5-100'.

\section{Combinatorics of nonsymmetric Macdonald polynomials}
\label{S:mac}

In this section we introduce the main combinatorial objects to be studied in this paper: affine root systems, alcove walks, the quantum Bruhat graph, and nonsymmetric Macdonald polynomials. The latter are best understood using the polynomial module for the double affine Hecke algebra, which we also discuss.

For more details on affine root systems and extended affine Weyl groups we refer the reader to \cite[Chapter 4]{Kac} and \cite[Chapters 1-2]{M3}. Our discussion of the nonsymmetric Macdonald polynomials and the double affine Hecke algebra follows \cite{OS}, though we also recommend the excellent references \cite{Ch2,H,M3,Sto-surv} on this subject.

\subsection{Affine root systems}
\label{SS:af}
Let $X$ be the weight lattice of a finite crystallographic root system $\De(X)$.\footnote{Following \cite{OS}, we find it convenient to use the lattice $X$ to label the root system $\De(X)$ and related objects, even though the root system is the most basic object. While it may seem excessive, we insist on applying the label $X$ frequently, because there will soon be an auxiliary lattice $Y$ and root system $\De(Y)$.} Let $\De_s(X)$ and $\De_l(X)$ be the subsets of $\De(X)$ consisting of all short and long roots, respectively. If $\De(X)$ is simply-laced, we take $\De_s(X)=\De_l(X)=\De(X)$. Let $\De_+(X)$ be a fixed set of positive roots in $\De(X)$ and let $\De_-(X)=-\De_+(X)$. Let $\{\al_i^X\}_{i\in I}\subset\De_+(X)$ be the corresponding simple roots and $\{\al_i^{\vee X}\}_{i\in I}\subset X^\vee:=\Hom(X,\Z)$ the simple coroots.

Let $\langle\cdot, \cdot\rangle: X^\vee\times X\to\Z$ be the evaluation pairing. Let $\{\om_i^X\}_{i\in I}\subset X$ be the fundamental weights given by $\pair{\al_j^{\vee X}}{\om_i^X}=\de_{ij}$. Let $Q^X=\bigoplus_{i\in I}\Z\al_i^X$ be the root lattice, $Q^X_+=\oplus_{i\in I}\Z_{\ge 0}\al_i^X$ the positive root cone, and $Q^X_-=-Q^X_+$. Let $X_+=\oplus_{i\in I}\Z_{\ge 0}\om^X_i$ be the cone of dominant weights and $X_- = -X_+$.

Let $W$ be the Weyl group of $\De(X)$. For each $\al\in\De(X)$, let $s_\al\in W$ denote the correspond reflection and let $\al^\vee\in\De(X)^\vee$ be the associated coroot. Let $s_i=s_{\al_i}$ for $i\in I$. The Weyl group $W$ acts on $X$ and $X^\vee$ by the formulas
\begin{align*}
s_\al(\la)=x-\pair{\al^\vee}{\la}\al,\quad s_\al(\la^\vee)=\la^\vee-\pair{\la^\vee}{\al}\al^\vee.
\end{align*}
For any $\la\in X$, let $\la_-$ be the unique element of $X_-$ in the $W$-orbit of $\la$.

An affinization of $\De(X)$ is an affine root system $\De(\tX)$ contained in the lattice $\tX=X\oplus\Z\de^X$. The affine root system is determined by a choice of the simple affine root $\al_0^X\in\tX$. We allow two possibilities for $\al_0^X$:
\begin{enumerate}
\item[$\ut$] $\al_0^X = -\theta^X +\de^X$, where $\theta^X$ is the dominant long root of $\De(X)$
\item[$\du$] $\al_0^X = -\vartheta^X + \de^X$, where $\vartheta^X$ is the dominant short root of $\De(X)$
\end{enumerate}
In $\ut$ we obtain the untwisted affinization of $\De(X)$ (Table Aff 1 of \cite{Kac}). When $\De(X)$ is simply-laced, one has $\theta^X=\vartheta^X$ and the two affinizations coincide. When $\De(X)$ is not simply-laced, $\theta^X\neq\vartheta^X$ and $\du$ yields the twisted affinization of $X$ (Tables Aff 2 and 3 of \cite{Kac}, excluding $A_{2n}^{(2)}$). We will call an affinization {\em dual untwisted} whenever $\du$ holds, both when $\De(X)$ is simply-laced and when it is not.\footnote{Such $\tX$ are exactly the dual affine root systems of untwisted affine root systems.}


An affinization determines an affine root system $\De(\tX)$ as follows. Let $\tI=I\sqcup\{0\}$. The simple affine roots are $\{\al_i^X\}_{i\in\tI}\subset\tX$. The set of all (real) affine roots is
\begin{align*}
\De(\tX)=\{\al+m\de^X : \al\in\De_s(X),m\in\Z\}\cup\{\al+m\de^X : \al\in\De_l(X),m\in r\Z\}
\end{align*}
where $r$ is the corresponding index of \cite[Table Aff $r$]{Kac}. The positive affine roots (determined by our choice of simple roots) are
\begin{align*}
\De_+(\tX)=\{\al+m\de^X\in\De(\tX) : \text{$\al\in\De(X)$, $m\ge 0$, where $m>0$ if $\al\in\De_-(X)$}\}.
\end{align*}
Let $\De_-(\tX)=-\De_+(\tX)$. For an affine root $\al+m\de^X\in\De(\tX)$, define
\begin{align*}
{\rm Re}(\al+m\de^X)&=\overline{\al+m\de^X}=\al,\quad
\deg(\al+m\de^X)=m.
\end{align*}

\subsection{Auxiliary affine root system}
Given a particular affinization $\De(\tX)$, we define an auxiliary affine root system $\De(\tY)$ as follows:
\begin{enumerate}
\item[$\ut$] If $\tX$ is untwisted, let $Y$ be the weight lattice of $\De(Y)=\De(X)^\vee$.\footnote{So, in this case, $Y$ is the coweight lattice of $\De(X)$.} Define $\De(\tY)$ to be the untwisted affinization of $\De(Y)$.
\item[$\du$] If $\tX$ is dual untwisted, let $Y$ be the weight lattice of $\De(Y)=\De(X)$ and define $\De(\tY)=\De(\tX)$.
\end{enumerate}
When $\De(X)$ is simply-laced, we identify $\De(X)$ with $\De(X)^\vee$ in the usual way and there is no difference between $\ut$ and $\du$.

The simple roots of $\De(\tY)$ are indexed compatibly with those of $\De(\tX)$:
\begin{enumerate}
\item[$\ut$] $\al_0^Y = -\theta^Y+\de^Y = -(\vartheta^X)^\vee +\de^Y$ and $\al_i^Y = \al_i^{\vee X}$ for all $i\in I$
\smallskip
\item[$\du$] $\al_i^Y = \al_i^X$ for all $i\in\tI$
\end{enumerate}
We transport all further notation from the previous section to $\De(\tY)$ in the obvious way.

\begin{rem}
While $\De(\tX)$ is the ``input'' datum for the nonsymmetric Macdonald polynomials, we will see below that all of the relevant combinatorics takes place in $\De(\tY)$.
\end{rem}

\begin{rem}
In \cite{OS}, the pair of lattices $(X,Y)$ is called a double affine datum, as such a pair gives rise to a double affine Hecke algebra.
\end{rem}

\subsection{Weyl groups}
Let $W = \langle s_i : i\in I\rangle$ be the (common) Weyl group of $\De(X)$ and $\De(Y)$. The extended affine Weyl groups of $\tX$ and $\tY$ are
\begin{align*}
W(\tX)&=Y\rtimes W,\qquad
W(\tY)=X\rtimes W.
\end{align*}
For any $\mu\in Y$, let $t_\mu=(\mu,\id)$ be the corresponding translation element in $W(\tX)$. We identify $W$ with a subgroup of $W(\tX)$ via $\si\mapsto (0,\si)$. The weight and direction of an element $w=(\la,\si)\in W(\tX)$ are $\wt(w)=\la$ and $\dir(w)=\si$. Hence $w=t_{\wt(w)}\dir(w)$ for any $w\in W(\tX)$. We transport these definitions to $W(\tY)$ in the obvious way.

The affine reflections corresponding to $\al_0^X$ and $\al_0^Y$ are
\begin{align*}
s_0^X=t_{-\vartheta^Y}s_{\vartheta_Y}\in W(\tX),\quad s_0^Y=t_{-\vartheta^X}s_{\vartheta^Y}\in W(\tY).
\end{align*}
When $\De(\tX)$ is untwisted, we have $s_0^X=t_{-(\theta^X)^\vee}s_{\theta^X}$.

The subgroups $W_a(\tX)=Q^Y\rtimes W = \langle s_0^X, W\rangle\subset W(\tX)$ and $W_a(\tY)=Q^X\rtimes W=\langle s_0^Y, W\rangle\subset W(\tY)$ are Coxeter groups of affine reflections. Let $\ell$ denote their length functions with respect to $\tS=\{s_i\}_{i\in\tI}$, where $s_0=s_0^X$ or $s_0=s_0^Y$.

Let $\Pi^X=W(\tX)/W_a(\tX)=Y/Q^Y$ and $\Pi^Y=W(\tY)/W_a(\tY)=X/Q^X$. One has the semidirect product decompositions
\begin{align*}
W(\tX)=\Pi^X\ltimes W_a(\tX),\qquad W(\tY)=\Pi^Y\ltimes W_a(\tY).
\end{align*}
Thus elements $w\in W(\tX)$ (resp. $w\in W(\tY)$) are uniquely represented as $w=\pi s_{i_1}\dotsm s_{i_\ell}$ where $i_j\in\tI$ and $s_0=s_0^X$ (resp. $s_0=s_0^Y$) and $\pi\in\Pi^X$ (resp. $\pi\in\Pi^Y$). We extend the length function on $W_a(\tX)$ to $W(\tX)$ by declaring $\ell(w)=\ell(\pi^{-1}w)$ for the unique element $\pi\in\Pi^X$ such that $\pi^{-1}w\in W_a(\tX)$. We obtain a length function on $W(\tY)$ similarly. Then an expression $w=\pi s_{i_1}\dotsc s_{i_\ell}$ is called reduced if $\ell=\ell(w)$.

The Bruhat order on the Coxeter system $(W(\tY),\tS)$ extends to a partial order $\le$ on $W(\tY)$ by declaring
\begin{align*}
\pi u \le \pi v :\Leftrightarrow u\le v
\end{align*}
for all $\pi\in\Pi^Y$ and $u,v\in W_a(\tY)$.\footnote{Thus, $\pi u$ and $\pi'v$ are incomparable if $\pi\ne\pi'$.} We will use the notation $\le$ for the compatible Bruhat order on $W$.

For any $\la\in X$, let $m_\la\in W(\tY)$ be the unique Bruhat minimal representative of the coset $t_\la W\in W(\tY)/W$. Explicitly, $m_\la = t_\la \si_\la$ where $\si_\la$ the unique minimal representative of the set $\{\si\in W : \si(\la)=\la_-\}$, which is a coset of the stabilizer of $\la$ in $W$. Thus, $\si_\la=\id$ if and only if $\la$ is antidominant, while $\si_\la=w_0$ if and only if $\la$ is dominant and regular. Here $w_0$ is the longest element of $(W,\{s_i : i \in I\})$.

Identifying $\la\in X$ with its representative $m_\la$, we obtain a partial order on $X$:
\begin{align}\label{E:X-order}
\mu\le\la :\Leftrightarrow m_\mu\le m_\la.
\end{align}
We note that $\mu\le\la$ implies $\mu_--\la_-\in Q_+$ and $\si_\mu>\si_\la$ if $\mu_-=\la_-$ (but not conversely).

\subsection{Pairing}
Let $(\cdot,\cdot) : X\times Y \to \Q$ be the $W$-invariant pairing specified by its restriction to $Q^X\times Q^Y\to \Z$ as follows:
\begin{enumerate}
\item[$\ut$] $Q^Y=Q^{\vee X}=\bigoplus_{i\in I}\Z\al_i^{\vee X}$, and $(\cdot,\cdot)$ is the canonical pairing between $Q^X$ and $Q^{\vee X}$ determined by the Cartan matrix of $\De(X)$, e.g, $(\al_i^X,\al_j^Y)=\pair{\al_j^{\vee X}}{\al_i^X}$ since $\al_j^Y=\al_j^{\vee X}$.
\item[$\du$] $Q^Y = Q^X$ and the pairing is uniquely determined by the normalization $(\al,\al)=2$ for all {\em short} $\al\in\De(X)$.
\end{enumerate}

Let $e$ be the smallest positive integer such that $(X,Y)\subset\frac{1}{e}\Z$. Then $W(\tX)$ acts on $X\oplus\frac{1}{e}\Z\de^X$ via the formula
\begin{align*}
(y,\si)\cdot(x+m\de^X) = \si(x)+(m-(\si(x),y))\de^X,
\end{align*}
where $x\in X$, $y\in Y$, $u\in W$, and $m\in\frac{1}{e}\Z$. The group $W(\tY)$ acts on $Y\oplus\frac{1}{e}\Z\de^Y$ by the same formula with the roles of $x$ and $y$ reversed. These actions preserve the sets of affine roots $\De(\tX)$ and $\De(\tY)$, respectively.

For any $w\in W(\tY)$ we define $\Inv(w)=\De_+(\tY)\cap w^{-1}(\Delta_-(\tY))$. Then we have $\ell(w)=|\Inv(w)|$. If $w=\pi s_{i_1}\dotsc s_{i_\ell}$ is a reduced expression, then $\Inv(w)=\{\be_1,\dotsc,\be_\ell\}$ where
\begin{align}\label{E:beta}
\be_j = s_{i_\ell}\dotsm s_{i_{j+1}}(\al_{i_j}).
\end{align}

\subsection{Alcove walks}
Suppose $w=\pi s_{i_1}\dotsm s_{i_\ell}\in W(\tY)$ is a reduced expression.

An {\em alcove walk} (or alcove path) of type $\uw=(\pi,i_1,\dotsc,i_\ell)$ beginning at $\si\in W(\tY)$ is a sequence $p=(\si_0,\si_1,\dotsc,\si_\ell)$ of elements of $W(\tY)$ satisfying
\begin{align*}
\si_0 = \si,\qquad \si_j \in \{\si_{j-1}, \si_{j-1}s_{i_j}\}\quad \text{for}\quad j\ge 1.
\end{align*}
The endpoint, weight, and direction of an alcove walk $p$ are defined as
\begin{align*}
\ept(p)=\si_\ell,\quad \wt(p)=\wt(\ept(p)),\quad \dir(p)=\dir(\ept(p)).
\end{align*}

\begin{rem}
The definition of alcove walks, due to \cite{Ram}, is essentially equivalent to that of LS-galleries from \cite{GL}. The $\la$-chains of \cite{LP} are also very closely related to alcove walks.
\end{rem}

For fixed $\uw$ and $\si$, there is a bijection between the set of alcove walks $p$ of type $\uw$ starting at $\si$ and the set of all subsets $J\subset\{1,\dotsc,\ell\}$. The bijection sends $p$ to the subset $J(p)=\{j : \si_j = \si_{j-1}\}$. The set $J(p)$ is called the set of folds of $p$. Given a subset $J\subset \{1,\dotsc,\ell\}$, we write $p_J$ for the corresponding alcove walk.

Let $\B(\si,\uw)$ be the set of all alcove walks of type $\uw$ starting at $\si$. Below we will write $\B(\si,w)$ for this set, with the understanding that the set depends on the choice of a reduced expression for $w$. We will also write $p_J\in\B(\si,w)$ to simultaneously specify the corresponding subset $J\subset\{1,\dotsc,\ell\}$.

Given $p_J\in\B(\si,w)$ where $J=\{j_1<\dotsm<j_r\}$ we may construct the sequence of elements $(z_0,z_1,\dotsc,z_r)$
\begin{align}\label{E:zs}
z_k = \si w s_{\be_{j_1}}\dotsm s_{\be_{j_k}} = \si\pi^Y s_{i_1}\dotsm \hat{s_{i_{j_1}}}\dotsm \hat{s_{i_{j_k}}}\dotsm s_{i_\ell}
\end{align}
where the $\beta_j$ are given by \eqref{E:beta}. We note that $z_0=\si w$ and $z_r=\ept(p_J)$.

Define the signs $\ep_j\in\{\pm 1\}$ for $j\in J$ by the condition
\begin{align*}
z_k(\be_{j_k})\in\ep_{j_k}\De_+(Y)+\Z\de^Y
\end{align*}
Define $J^+=\{j\in J : \ep_j = 1\}$ and $J^-=\{j\in J : \ep_j = -1\}$. These are called the subsets of positive and negative folds, respectively.

\subsection{Quantum Bruhat graph}
\label{SS:QBG}
The quantum Bruhat graph $\QBG(\tX)$ is the directed graph with vertex set $W$ and edges $w\xrightarrow{\al} ws_\al$ labeled by $\al\in \Delta_+(Y)$ whenever either of the following holds:
\begin{itemize}
\item $\ell(ws_\al)=\ell(w)+1$ (Bruhat edges)
\item $\ell(ws_\al)=\ell(w)-\pair{2\rho^{\vee Y}}{\al}+1$ (quantum edges)
\end{itemize}
where $\rho^{\vee Y}=\frac{1}{2}\sum_{\al\in\De_+(Y)} \al^\vee$. We observe that
\begin{align}\label{E:QBG-simple}
\text{$w\xrightarrow{\al}ws_\al$ and $w\xleftarrow{\al}ws_\al$ if and only if $\al\in\De(Y)$ is simple.}
\end{align}

\begin{rem}
The quantum Bruhat graph was first defined (in the untwisted setting) in \cite{BFP}, where it was used in the study of the (small) quantum cohomology of flag varieties. As a matter of fact, the edges in $\QBG(\tX)$ are compatible with an order of Lusztig \cite{lu} on $W(\tX)$, projected to $W$ via the map $\dir$. The definition of $\QBG(\tX)$ given above is from \cite[\S4.1]{OS}, where it is also extended to the affine root systems $A_{2n}^{(2)}$ and $A_{2n}^{(2)\dagger}$.
\end{rem}

Let $\QB(\si,w)\subset\B(\si,w)$ be the set of all alcove walks $p_J\in \B(\si,w)$ such that, for the elements $(z_0,z_1,\dotsc,z_r)$ from \eqref{E:zs}, we have
\begin{align*}
\dir(z_0)\xrightarrow{-\bbe_1}\dir(z_1)\xrightarrow{-\bbe_2}\dotsm\xrightarrow{-\bbe_r}\dir(z_r)
\end{align*}
as a directed path in $\QBG(\tX)$. Let $\QBR(\si,w)\subset\B(\si,w)$ be the set of all $p_J\in\B(\si,w)$ such that
\begin{align*}
\dir(z_0)\xleftarrow{-\bbe_1}\dir(z_1)\xleftarrow{-\bbe_2}\dotsm\xleftarrow{-\bbe_r}\dir(z_r)
\end{align*}
is a directed path in $\QBG(\tX)$.

\subsection{Nonsymmetric Macdonald polynomials}

For any (semi-)ring $R$, we define the group (semi-)ring $R[X] = R[X^\la : \la\in X]$, where $X^\la X^\mu=X^{\la+\mu}$. Elements of $R[X]$ are finite linear combinations of monomials $X^\la$ with coefficients in $R$. For any $f\in R[X]$, we denote by $[X^\la]f\in R$ the coefficient of $X^\la$ in $f$.

Let $\K=\Q(q,v)$. The nonsymmetric Macdonald polynomials $\{E_\la(X;q,v)\}_{\la\in X}$ attached to the affine root system $\De(\tX)$ form a basis for $\K[X]$.\footnote{We will need the equal parameter specializations of the $E_\la$, i.e., we set all $v$-parameters equal to a single parameter $v$.} They were first defined in \cite{M2}, building on the works \cite{Ch-ann,O}. While coefficients of $E_\la(X;q,v)$ actually belong to $\Q(q,v^2)$, it is convenient for our purposes to work in terms of the parameter $v$. The $E_\la$ are lower triangular with respect to the monomial basis $\{X^\la\}_{\la\in X}$ of $\K[X]$:
\begin{align}
E_\la(X;q,v) = X^\la + \sum_{\mu\leq\la} c_{\la\mu}(q,v)X^\mu
\end{align}
with respect to the partial order on $X$ defined in \eqref{E:X-order} above.

\subsection{DAHA polynomial module}
Let $\LL=\Q(q^{\frac{1}{2e}},v)$, where $q^{\frac{1}{2e}}$ is a parameter such that $(q^{\frac{1}{2e}})^{2e}=q$. The group algebra $\LL[X]$ carries an important action of the double affine Hecke algebra (DAHA), which one may use to characterize and study the $E_\la(X;q,v)$.

The DAHA is
\begin{align*}
\HH = \LL[X]\otimes \H(W)\otimes\LL[Y]
\end{align*}
where $\H(W)=\oplus_{\si\in W} \LL T_\si$ is the Hecke algebra of $W$, all the tensor products are over $\LL$, and the tensor factors are subalgebras. The tensor factors do not commute with each other (i.e., the ``definition'' of $\HH$ above holds on the level of $\LL$-vector spaces, but not $\LL$-algebras), but the relations between them can be given explicitly (see, e.g., \cite[Definition~3.2.1]{Ch2}). The generators $\{T_i\}_{i\in I}$ of the Hecke algebra $\H(W)$ satisfy the braid relations for $W$ and the quadratic relations
\begin{align*}
(T_i-v)(T_i+v^{-1})=0.
\end{align*}
The $\LL$-basis elements of $\H(W)$ are defined as $T_\si=T_{i_1}\dotsm T_{i_\ell}$ for any reduced expression $\si=s_{i_1}\dotsm s_{i_\ell}$, and they are independent of this choice.

The algebra structure on $\HH$ can be described in terms of its faithful action on the polynomial module $\LL[X]$. Abstractly, $\LL[X]$ is the $\HH$-module induced from the trivial character of the affine Hecke subalgebra $\H(W)\otimes\LL[Y]$.
Explicitly, the generators of $\H(W)$ act on $\LL[X]$ by Demazure-Lusztig operators
\begin{align}\label{E:DL}
T_i = vs_i + \frac{v-v^{-1}}{X^{\al_i}-1}(s_i-1),
\end{align}
the $X^\la$ act by multiplication operators, and the $Y^\mu$ act by difference-reflection Dunkl operators. The important connection to Macdonald polynomials is that the $Y^\mu$ act diagonally\footnote{Here we see the need for the extension $\LL$ of the field $\K$. Technically, we only need $\Q(q^{1/e},v)$ at the moment, but later we will need the $2e$-th root of $q$.} on the $E_\la$ (and this gives one characterization of the $E_\la$):
\begin{align*}
Y^\mu\cdot E_\la(X;q,v) = q^{-(\la,\mu)}v^{2\pair{\si_\la^{-1}(\rho^{\vee Y})}{\mu}}E_\la(X;q,v)
\end{align*}
for all $\mu\in Y$ and $\la\in X$.

Generally, the elements of $\HH$ act on $\LL[X]$ by difference-reflection operators. These are built out of the natural action of $W(\tX)=Y\rtimes W$ on $\LL[X]$ by
\begin{align*}
(\mu,\si)\cdot X^\la = q^{-(\si(\la),\mu)}X^{\si(\la)}.
\end{align*}
(If we set $X^{\de^X}=q$, then this is nothing but $X^{(\mu,\si)\cdot(\la+0\de^X)}$.) For any $\mu\in Y$, let $\Ga_\mu$ be the $q$-shift operator
\begin{align*}
\Ga_\mu\cdot X^\la = (-\mu,\id)\cdot X^\la = q^{(\la,\mu)}X^\la.
\end{align*}
A difference-reflection operator is any element of $\End(\LL[X])$ given as a finite sum
\begin{align*}
\sum_{\si\in W}\sum_{\mu\in Y} f_{\mu,\si}(X;q,v)\,\Ga_\mu\,\si,
\end{align*}
where $f_{\mu,\si}(X;q,v)\in\LL(X):=\Frac(\LL[X])$, the field of rational functions acting by multiplication. Not all such operators preserve $\LL[X]$, but those from the polynomial module of $\HH$ do (e.g., \eqref{E:DL}).

We will need an additional structure on the polynomial module, namely the $\Q$-linear involution $*$ defined by $(X^\la)^*=X^{-\la}$, $(q^{1/(2e)})^*=q^{-1/(2e)}$, and $v^*=v^{-1}$. One may directly check that for any $f\in\LL[X]$ and $\si\in W$, we have
\begin{align}\label{E:*-T}
(T_\si\cdot f)^* = T_{\si^{-1}}^{-1}\cdot f^*.
\end{align}
See \cite[(3.2.19)]{Ch2} for this and the corresponding formulas for other elements of $\HH$.

\subsection{Ram-Yip formula}

An important insight of Ram and Yip \cite{RY} is that the action of DAHA on $\LL[X]$---in particular, the intertwiner construction of the $E_\la$ from \cite[Proposition~3.3.5]{Ch2}---can be exploited to both discover and prove a combinatorial formula for the $E_\la$ attached to any affine root system $\De(\tX)$.

The technique of \cite{RY} gives a combinatorial formula for the following more general elements of $\K[X]$. For any $\la\in X$ and $\si\in W$, let us define
\begin{align}\label{E:def-E-u}
E_\la^\si(X;q,v) &= v^{-\ell(\si\si_\la^{-1})+\ell(\si_\la^{-1})}T_\si\cdot E_\la(X;q,v).
\end{align}
In particular, $E_\la^\id(X;q,v)=E_\la(X;q,v)$ and by \cite[(3.3.24)]{Ch2}
\begin{align}\label{E:Tw0}
E_\la^{w_0}(X;q,v)=E_{-w_0\la}(X;q,v)^*.
\end{align}
For any fixed $\si\in W$, the polynomials $\{E_\la^\si(X;q,v)\}_{\la\in X}$ form a basis for the polynomial module (since $T_\si$ is invertible). We will see below that the $E_\la^\si$ are well-defined at $v=0$ and $v=\infty$. In type $A$, the $E_\la^\si$ are exactly the general basement nonsymmetric Macdonald polynomials studied in \cite{A}.\footnote{Note however that \cite{A} indexes these polynomials by $\si^{-1}$.}

\begin{rem}
Up to a power of $v$, the polynomial $E_\la^\si$ depends only on the pair of weights $(\la,\si(\la))$. However, we still prefer to separate $\si$ from $\la$ in the notation.
\end{rem}

The following is the Ram-Yip formula for the polynomials $E_\la^\si(X;q,v)$.

\begin{thm}[\cite{RY}]
For any $\la\in X$ and $\si\in W$,
\begin{align}\label{E:RY}
E_\la^\si(X;q,v) &= v^{-2\ell(\si\si_\la^{-1})}\sum_{p_J\in\B(\si,m_\lambda)}v^{\ell(\dir(p_J))}X^{\wt(p_J)}\prod_{j\in J}\frac{v^{-1}-v}{1-\xi_j}\prod_{j\in J^-}\xi_j
\end{align}
where $\xi_j=q^{\deg(\be_j)}v^{-\pair{2\rho^{\vee Y}}{\bbe_j}}$, for $\be_j$ defined in \eqref{E:beta}.
\end{thm}

The Ram-Yip formula immediately gives the following:

\begin{cor}\label{C:TE-supp}
We have $[X^{\si(\la)}]E_\la^\si=1$, and $[X^\mu]E_\la^\si\neq 0$ only if $\si^{-1}(\mu)\leq\la$. In particular, $[X^{\si(\la_-)}]E_\la^\si=\de_{\la,\la_-}$.
\end{cor}
\begin{proof}
If $p\in\B(\si,m_\la)$, then $\wt(p)=\si(\nu)$ where $\nu\le\la$. Hence $\si^{-1}(\wt(p))=\nu\le\la$. If $\wt(p)=\si(\la_-)$ for some $p$, then $\la_-=\si^{-1}(\wt(p))\le \la\le\la_-$. This arises only when $\la=\la_-$ and $p=p_\emptyset$ (for then $\wt(p_J)<\la_-$ whenever $J\neq\emptyset$).
\end{proof}

Define $\qwt(p_J) = \sum_{j\in J^-}\be_j$.

\begin{prop}[{\cite[Corollary~4.4]{OS}}]
\label{P:RY-limit-0}
For any $\la\in X$ and $\si\in W$,
\begin{align}\label{E:RY-limit-0}
E_\la^\si(X;q,0) &= \sum_{p\in \QB(\si,m_\la)} q^{\deg(\qwt(p))}X^{\wt(p)}.
\end{align}
\end{prop}

Below we will also need the existence of $E_\la^\si$ at $v=\infty$. We now derive a formula analogous to \eqref{E:RY-limit-0} at $v=\infty$, which shows in particular that the $E_\la^\si$ are well-defined at this specialization. By $v=\infty$, we mean that all coefficients are expressed in terms of $v^{-1}$ and then we set $v^{-1}=0$. The special case $\si=\id$ of the following formula was proved in \cite[Proposition~5.4]{OS}.

Define $\qwt^*(p_J)=\sum_{j\in J^+}\be_j$.

\begin{prop}\label{P:RY-limit-inf}
For any $\la\in X$ and $\si\in W$,
\begin{align}
E_\la^\si(X;q^{-1},\infty) &= \sum_{p\in \QBR(\si,m_\la)} q^{\deg(\qwt^*(p))}X^{\wt(p)}.
\end{align}
\end{prop}

\begin{proof}
Let $\ord : \K\to\Z$ denote the order with respect to $v^{-1}$. For $p\in\B(\si,m_\la)$, let $\ord(p)$ be the order of the corresponding term in the Ram-Yip formula (i.e., the order of its coefficient, which belongs to $\Q(q,v)$). For the unfolded walk, we have $\ord(p_\emptyset)=\ell(\si_\la^{-1})-\ell(\si\si_\la^{-1})$.

Now suppose $J\subsetneq\{1,\dotsc,\ell\}$ and $\hat{J}=J\cup\{M\}$, with $M$ being the maximal element of $\hat{J}$. As shown in \cite[Proof of Corollary~4.4]{OS}, we have $M\in\hat{J}^+$ iff $\dir(p_{\hat{J}})>\dir(p_J)$, and $M\in\hat{J}^-$ iff $\dir(p_{\hat{J}})<\dir(p_J)$.

Suppose $M\in\hat{J}^+$. Then
\begin{align*}
\ord(p_{\hat{J}})-\ord(p_J)=(\ell(\dir(p_{J}))-\ell(\dir(p_{\hat{J}})))-1+\pair{2\rho^{\vee Y}}{-\bbe_M}\ge 0
\end{align*}
since $\ell(s_{-\bbe_M})\le \pair{2\rho^{\vee Y}}{-\bbe_M}-1$.

Suppose $M\in\hat{J}^-$. Then
\begin{align*}
\ord(p_{\hat{J}})-\ord(p_J)=(\ell(\dir(p_{J}))-\ell(\dir(p_{\hat{J}})))-1\ge 0.
\end{align*}

By induction on $|J|$, it follows that the $v=\infty$ specialization of $E_\la^\si$ is well-defined. A term on from the Ram-Yip formula survives in this limit iff the corresponding inequality above is an equality for each fold in the alcove walk. These are exactly the paths in $\QBR(\si,m_\la)$. The power of $q$ is easily found.
\end{proof}

\begin{rem}
It is also possible to deduce Proposition~\ref{P:RY-limit-inf} from Proposition~\ref{P:RY-limit-0} using \eqref{E:Tw0}, as is done in the proof of \cite[Proposition~5.4]{OS}.
\end{rem}

\section{Generalized global Weyl modules}
\label{S:weyl}

Now we turn to Weyl modules. First we recall some known results on classical Weyl modules (local and global) and generalized local Weyl modules. Our main results on generalized global Weyl modules are stated and proved in \S\ref{SS:glob} and \S\ref{SS:struct}.

\subsection{Classical Weyl modules}
Let $\fg$ be a finite-dimensional simple Lie algebra over $\C$ with Cartan decomposition $\fg=\fn_+\oplus\fh\oplus \fn_-$.
We denote by $\Delta$ the set of roots and by $\Delta_+$ (resp. $\Delta_-$) the subsets of positive (resp. negative)
roots of $\fg$ with respect to the Cartan subalgebra $\fh$. For $\al\in\Delta_+$ we denote by $e_\al\in\fn_+$
and $f_{-\al}\in\fn_-$ the corresponding root vectors. Let $n=\dim\fh$ be the rank of $\fg$
and let $\al_i$, $\al_i^\vee$, $\om_i$, $i=1,\dots,n$, be the simple roots, simple coroots  and fundamental weights.
In what follows we use the shorthand notation $f_i=f_{-\al_i}$, $e_i=e_{\al_i}$.
The commutators $h_i=[e_i,f_i]$, $i=1,\dots,n$ span the Cartan subalgebra $\fh$.

We denote by $Q$ the root lattice, generated by simple roots, and by $X$ the weight lattice, generated by
the fundamental weights. In particular, $X$ contains $X_+$, the set of dominant integral weights.
Each $\la=\sum_{i=1}^n m_i\om_i\in X_+$, $m_i\in\bZ_{\ge 0}$ is the highest weight of an irreducible
$\fg$ module $V_\la$.

For $\la\in X_+$, let $W(\la)$ and $\W(\la)$ be the local and global Weyl module of highest weight $\la$ \cite{CL,FL1,FL2,CFK}.
Both $W(\la)$ and $\W(\la)$ are cyclic representations of the current
algebra $\fg\T\bC[t]$ with a cyclic vector $v$ of weight $\la$ ($(h\T 1)v=\la(h)v$ for all $\h\in\fh$).
The defining relations of $\W(\la)$ are
\[
(\fn_+\T t\C[t])v=0,\qquad
(f_\alpha\otimes 1)^{-\langle \alpha^\vee, \lambda \rangle+1} v=0 \quad \text{for all $\alpha \in \Delta_-$}.
\]
To define the local Weyl module $W(\la)$ one adds the extra relation $(\fh\T t\bC[t])v=0$.
Let $A(\la)={\rm U}(\fh\T t\bC[t])v\subset \W(\la)$ be the space of highest weight vectors.
The space $A(\la)$ can be identified with the quotient algebra
of the polynomial ring ${\rm U}(\fh\T t\bC[t])$ modulo the ideal of polynomials vanishing on $v$.
Hence $A(\la)$ inherits the natural grading with respect to the degree of $t$.
An important observation (see \cite{CP,CFK}) is that the global Weyl module can be equipped with the structure of
a bimodule over $\fg\T\bC[t]\times A(\la)$, where the action of an element $a\in A(\la)$ on a vector $rv\in \W(\la)$,
$r\in {\rm U}(\fg\T\bC[t])$, is given by $r(av)$.

Let $\la=\sum_{i=1}^n m_i\om_i$.
In order to describe $A(\la)$, we consider the ring $B(\la)$ of polynomials in the variables
$x^i_j$ ($i=1,\dots,n$ and $j=1,\dots,m_i$) which are invariant under the natural action of the product of the symmetric
groups $S_{m_1}\times\dots\times S_{m_n}$  ($S_{m_i}$ permutes the variables $x^i_j$ with fixed $i$).
Hence $B(\la)=\otimes_{i=1}^n \bC[x_1^i,\dots,x_{m_i}^i]^{S_{m_i}}$.
The last piece of notation we need is
\[
(q)_\la=\prod_{i=1}^n (q)_{m_i},\quad (q)_m=\prod_{j=1}^m (1-q^j).
\]

The following fundamental results can be found in \cite{CFK}.
\begin{prop}
The assignment
\[
h_i\T t^k\mapsto 1\otimes\dotsm\otimes 1\otimes ((x^i_1)^k+\dotsm +(x^i_{m_i})^k)\otimes 1\otimes\dots\otimes 1
\]
induces an isomorphism of graded rings between $A(\la)$ and $B(\la)$.  In particular,
$\ch A(\la)=1/(q)_\la.$
\end{prop}

\begin{thm}\label{loc-glob-classical}
The algebra $A(\la)$ acts freely on the global Weyl module $\W(\la)$. The quotient $\W(\la)/A(\la)$
is isomorphic to the local Weyl module $W(\la)$. In particular,
\[
\ch \W(\la)=\frac{\ch W(\la)}{(q)_\la}.
\]
\end{thm}

\subsection{Generalized local Weyl modules}
In this subsection we recall the main results of \cite{FM3}.
Let $\widehat\fg$ be the untwisted affine Kac-Moody Lie algebra with the set of real roots given by
$\De(\tX)=\{\al+k\delta : \al\in\Delta, k\in\bZ\}$, where $\delta=\delta^X$ is the basic imaginary root.
The level-zero weight lattice of $\widehat\fg$ is $\tX=X\oplus\Z\de$. Here we follow the notation of \S\ref{SS:af}, except that we now frequently choose to omit the superscript $X$. For instance, we write $\al_i^\vee=\al_i^Y=\al_i^{\vee X}$ for $i=0,1,\dotsc,n$ (we take this as the definition of $\al_0^{\vee X}$ when $i=0$). We write simply $\QBG$ for $\QBG(\tX)$ from \S\ref{SS:QBG}.
Since we are the in the untwisted setting, $\tY$ is the level-zero weight lattice of the untwisted affinization of $\Delta(X)^\vee$. We let $\tW=W(\tY)=X\rtimes W$ be the extended affine Weyl group of $\tY$.

Recall that for an element $\mu\in X$ we denote by $t_\mu\in \tilde W$ the corresponding translation element in the extended affine Weyl group.
For $k=1,\dots,n$, let us fix a reduced expression $t_{-\om_k}=\pi s_{i_1}\dots s_{i_r}$
and consider the set of affine coroots $\beta^k_m\in\De(\tY)$, $m=1,\dots,r$, from \eqref{E:beta}:
\[
\beta^k_m=s_{i_r}\dots s_{i_{m+1}}\al^\vee_{i_m}.
\]
In particular, the real parts ${\rm Re}\,\be_m^k$, are negative coroots for $\fg$ (and ${\rm Re}\,\be_1^k=-\al^\vee_k$ is the
negative simple coroot).

Let us consider the nilpotent subalgebra $\fn^{af}=\fn_+\T 1\oplus\fg\T t\bC[t]\subset \widehat\fg$.
For $\sigma\in W$ and $\la\in X_-$ the generalized local Weyl module $W_{\sigma(\la)}$ is the cyclic $\fn^{af}$-module defined by relations
$\fh\T t\bC[t] v=0$ and
\begin{align*}
(f_{\alpha}\otimes t) v=0,&\quad \forall\,\alpha \in \sigma(\Delta_-)\cap \Delta_-;\\
(e_\alpha\otimes 1) v=0,&\quad \forall\,\alpha \in \sigma(\Delta_-)\cap\Delta_+;\\
(f_{\sigma(\alpha)}\otimes t)^{-\langle \alpha^\vee, \lambda \rangle+1} v=0,&\quad \forall\,\alpha \in \Delta_+\cap\si^{-1}(\Delta_-);\\
(e_{\sigma(\alpha)}\otimes 1)^{-\langle \alpha^\vee, \lambda \rangle+1} v=0,&\quad \forall\,\alpha \in \Delta_+\cap\si^{-1}(\Delta_+).
\end{align*}

\begin{rem}\label{R:hwt}
Let us extend the current algebra $\fg\T\bC[t]$ with the element $d\in \widehat\fg$ satisfying $[d,x\T t^k]=kx\T t^k$.
Then the action of $\fn^{af}$ on the modules $W_{\si(\la)}$ can be extended to $\fn^{af}\oplus\fh\otimes 1\oplus \bC d$ by imposing the
additional relations $(h\otimes 1)v = \mu(h)(v)$, $dv=lv$ for any $\mu\in X$, $l\in\bZ$. Denote the resulting module by
$W_{\si(\la)}^{\mu,l}$. We note that the pair $(\mu,l)$ is nothing but the affine weight $\mu+l\delta$.
We stress that $(\mu,l)$ can be arbitrary here, in contrast to the finite-dimensional setting. In particular, we have
$W_{\si(\la)}^{\mu,l}\simeq W_{\si(\la)}$ as $\fn^{af}$-modules for any $\mu\in X$, $l\in\bC$. However, unless otherwise specified,
we will denote by $W_{\si(\la)}$ either the $\fn^{af}$-module $W_{\si(\la)}$ or the $\fn^{af}\oplus\fh\otimes 1\oplus \bC d$-module
$W_{\si(\la)}^{\si(\la),0}$, so that the cyclic vector $v$ has weight $(\si(\la),0)$.
\end{rem}

For an element $\sigma \in W$ and $\alpha \in \Delta_+$ we set
\begin{align*}
\widehat \sigma(f_{-\alpha}\T t)&=
                            \begin{cases}
                             f_{-\sigma(\alpha)}\T t,&\ \text{if $\sigma(\alpha) \in \Delta_+$} \\
                             e_{-\sigma(\alpha)}\otimes 1,&\ \text{if $\sigma(\alpha) \in \Delta_-$} \\
                            \end{cases}\\
\widehat \sigma(e_{\alpha}\T 1)&=
                            \begin{cases}
                             e_{\sigma(\alpha)}\T 1,&\ \text{if $\sigma(\alpha) \in \Delta_+$} \\
                             f_{\sigma(\alpha)}\otimes t,&\ \text{if $\sigma(\alpha) \in \Delta_-$}. \\
                            \end{cases}
\end{align*}														
This action is compatible with the following action on a subset of affine roots:
\begin{align*}
\widehat \sigma(-\alpha+\delta)&=   \begin{cases}
                             {-\sigma(\alpha)+\delta},&\ \text{if $\sigma(\alpha) \in \Delta_+$} \\
                             {-\sigma(\alpha)},&\ \text{if $\sigma(\alpha) \in \Delta_-$} \\
                            \end{cases}\\
\widehat \sigma(\alpha)&=   \begin{cases}
                             {\sigma(\alpha)},&\ \text{if $\sigma(\alpha) \in \Delta_+$} \\
                             {\sigma(\alpha)}+ \delta,&\ \text{if $\sigma(\alpha) \in \Delta_-$}. \\
                            \end{cases}										
\end{align*}
Using this notation, the generalized local Weyl modules are defined by three sets of relations:
$(\fh\T t\bC[t]) v=0$, $\widehat\sigma(f_{-\alpha}\otimes t) v=0$, and
$\widehat\sigma(e_{\alpha}\otimes 1)^{-\langle \alpha^\vee, \lambda \rangle+1} v=0$ ($\alpha \in \Delta_+$).

In order to study the modules $W_{\si(\la)}$, we introduce the generalized local Weyl modules with characteristics.
Suppose that $\la\in X_-$ and $\langle\al_i^\vee,\la\rangle<0$ for some $i=1,\dots,n$.
We fix a reduced decomposition of the element $t_{-\om_i}$  and the corresponding sequence of
affine coroots $\beta^i_\bullet=\{\beta^i_1,\dots,\beta^i_r\}$.
The generalized local Weyl module with characteristics $W_{\sigma(\la)}(\beta^i_\bullet,m)$ is the
cyclic $\fn^{af}$-module defined by the relations
$(\fh\T t\bC[t])v=0$ and
\begin{align*}
\widehat{\sigma}(f_{-\alpha}\T t)v=0,\quad \widehat{\sigma}(e_{\alpha})^{l_{\al,m}+1}v=0 \quad \text{for all $\al\in\Delta_+$},
\end{align*}
where $l_{\al,m}= -{\langle \alpha^\vee, \lambda \rangle-|\lbrace \beta^i_l : {\rm Re}\, \beta^i_l=-\alpha^\vee, l \leq m \rbrace|}.$
For example, $W_{\sigma(\la)}(\beta^i_\bullet,0)\simeq W_{\sigma(\la)}$ and $W_{\sigma(\la)}(\beta^i_\bullet,r)\simeq W_{\sigma(\la+\om_i)}$.
In what follows we omit $\beta^i_\bullet$ and write $W_{\sigma(\la)}(m)$ instead of $W_{\sigma(\la)}(\beta^i_\bullet,m)$
(we always assume that a reduced decomposition of $t_{-\om_i}$ is fixed).
Let us consider the chain of natural surjections
\begin{equation}\label{chain}
W_{\sigma(\la)}= W_{\sigma(\la)}(0)\to W_{\sigma(\la)}(1)\to W_{\sigma(\la)}(2)\to\dots\to W_{\sigma(\la)}(r)=W_{\sigma(\la+\om_i)}.
\end{equation}
\begin{thm}[{\cite[Theorem 2.18]{FM3}}]  \label{localmain}
The surjection $W_{\sigma(\la)}(m)\to W_{\sigma(\la)}(m+1)$ is an isomorphism if and only if there is no edge
$\sigma \to \sigma s_{{\rm Re}\be^i_{m+1}}$ in $\QBG$. Assume that such an edge does exist and let
${\rm Re}\,\be^i_{m+1}=-\al_i^\vee$. Then the kernel of the surjection
is isomorphic to $W^{\si(\la)+(l_{\al,m}+1)\widehat{\sigma}(\al)}_{\sigma s_{{\rm Re}\be^i_{m+1}}(\la)}(m+1)$.
\end{thm}
\begin{rem}\label{shift}
The weight $(l_{\al,m}+1)\widehat{\sigma}(\al)$ may have a nontrivial $\delta$ component---see Remark \ref{R:hwt}.
More precisely, the cyclic vector of the kernel $W_{\sigma(\la)}(m)\to W_{\sigma(\la)}(m~+~1)$ is equal to
$(\widehat{\sigma} e_\al)^{l_{\al,m}+1}v$, with $v$ being the cyclic vector of $W_{\sigma(\la)}(m)$, and hence the
shift of the weight (the upper index).
Theorem \ref{localmain} has an obvious generalization. Namely, if one starts with a surjection
$W^{\nu}_{\sigma(\la)}(m)\to W^\nu_{\sigma(\la)}(m+1)$ for an affine weight $\nu$, then the kernel (if nontrivial) is isomorphic to
$W^{\nu+\sigma(\la)+(l_{\al,m}+1)\widehat{\sigma}(\al)}_{\sigma s_{{\rm Re}\be^i_{m+1}}(\la)}(m+1)$.
\end{rem}

Theorem~\ref{localmain} gives a filtration on the generalized local Weyl modules with subquotients isomorphic to generalized local Weyl modules
$W^\nu_{\tau(\la)}$, where $\nu$ are some affine weights and $\tau\in W$ are all possible ends of the paths in $\QBG$ starting at
$\sigma$ and following the edges with labels prescribed by the roots ${\rm Re}\,\beta^i_m$. More precisely,
all the paths are of the form
\[
\si\to \si s_{{\rm Re}\be^i_{m_1+1}}\to \si s_{{\rm Re}\be^i_{m_1+1}}s_{{\rm Re}\be^i_{m_2+1}}\to\dots
\]
for a sequence $0\le m_1<m_2<\dots <r$.

The filtration is constructed as follows: first, let us consider the kernel $K_{0,r}$ of the map
$W_{\sigma(\la)}\to W_{\sigma(\la)}(r)=W_{\sigma(\la+\om_i)}$. The space $K_{0,r}$ contains the kernel
$K_{0,r-1}$ of the map $W_{\sigma(\la)}\to W_{\sigma(\la)}(r-1)$. Thanks to
Theorem \ref{localmain} we have $K_{0,r}/K_{0,r-1}\simeq W_{\sigma s_\gamma(\lambda)}(r)$ for some $\gamma$.
Moving backwards along the chain \eqref{chain} we obtain a filtration on $W_{\sigma(\la)}$ with subquotients isomorphic
to $W_{\tau(\la)}(m)$, $\tau\in W$, $m\ge 1$. We then proceed further with each $W_{\tau(\la)}(m)$ using \eqref{chain}.
Since $W_{\tau(\la)}(r)\simeq W_{\tau(\la+\om_i)}$ for all $\tau$, Theorem \ref{localmain} indeed produces
the desired filtration.

\subsection{Generalized global Weyl modules}
\label{SS:glob}
We define generalized global Weyl modules $\W_\mu$, depending on an arbitrary $\mu\in X$.
The generalized global Weyl modules are cyclic representations of the algebra
$\fn^{af}=\fg\T t\bC[t]\oplus\fn_+\T 1$ defined by the following of relations.
Namely, let $\mu=\sigma(\la)$ for an antidominant $\la\in X_-$ and $\sigma\in W$. Then the relations in  $\W_{\sigma(\la)}$ are
as follows ($v$ is the cyclic vector):
\begin{align*}
(f_{\alpha}\otimes t) v=0,&\quad \forall\,\alpha \in \sigma(\Delta_-)\cap \Delta_-;\\
(e_\alpha\otimes 1) v=0,&\quad \forall\,\alpha \in \sigma(\Delta_-)\cap\Delta_+;\\
(f_{\sigma(\alpha)}\otimes t)^{-\langle \alpha^\vee, \lambda \rangle+1} v=0,&\quad \forall\,\alpha \in \Delta_+\cap\si^{-1}(\Delta_-);\\
(e_{\sigma(\alpha)}\otimes 1)^{-\langle \alpha^\vee, \lambda \rangle+1} v=0,&\quad \forall\,\alpha \in \Delta_+\cap\si^{-1}(\Delta_+).
\end{align*}
We can write the relations in a more compact form using the $\widehat\sigma$ notation:
\[
\widehat\si (f_{-\al}\T t)v=0, \ \widehat\si (e_{\al}\T 1)^{-\langle \al^\vee,\la\rangle +1}v=0,\quad \text{for all $\al\in\Delta_+$}.
\]

\begin{rem}\label{R:hwt-glob}
As in Remark~\ref{R:hwt}, we can extend the action of $\fn^{af}$ on $\W_{\si(\la)}$ to $\fn^{af}\oplus\fh\otimes 1\oplus\bC d$
by declaring that $v$ has $\fh$-weight $\mu$ and $d$-weight $l$ for any $\mu\in X$, $l\in \bZ$. It is convenient to identify
the pair $(\mu,l)$ with the affine weight $\nu=\mu+l\delta$.
We denote the resulting module by $\W_{\si(\la)}^\nu$, and continue to use $\W_{\si(\la)}$ to denote the module
$\W_{\si(\la)}^{\si(\la)+0\delta}$.
Then $\W(w_0\la)\simeq\W_\la$ as $\fn^{af}\oplus\fh\otimes 1$-modules for any $\la\in X_-$.
\end{rem}

The modules $\W_{\si(\la)}^\nu$ are naturally graded by $\fh$ and by $d$. We say that a vector $v\in \W_{\si(\la)}^\nu$ has
degree $j$ if $dv=jv$.
For a $d$-graded $\mathfrak{h}$-module $V=\oplus_{j\in\bZ} V_j$ with $\dim V_j <\infty$ for all $j$, we define the
graded character as an element of $\Z_{\ge 0}[q][X]$:
\begin{align}
\ch V = \sum_{j\in\bZ} q^j\sum_{\la\in P}(\dim\,(V_j)_\la) X^\la.
\end{align}

Let $A_{\sigma(\la)}={\rm U}(\fh\T t\bC[t])v\subset \W_{\sigma(\la)}$ be the space of the weight $\sigma(\la)$ vectors and let
$\la=-\sum_{i=1}^n m_i\om_i$, $m_i\ge 0$. The space $A_{\sigma(\la)}$ is naturally a quotient
of the polynomial ring ${\rm U}(\fh\T t\bC[t])$ and hence carries the natural grading with respect to the degree of $t$.

\begin{lem}\label{bi-module}
The algebra $A_{\sigma(\la)}$ acts on $\W_{\sigma(\la)}$ as follows: for $a\in A_{\sigma(\la)}$ and $r\in {\rm U}(\fn^{af})$
one has $(rv). a=r(av)$.
\end{lem}
\begin{proof}
The proof given in \cite[\S3.4]{CFK} carries over without difficulty. (By Remark~\ref{R:hwt-glob}, the argument given in {\em op. cit.} corresponds to the case $\si=\id$.)
\end{proof}

\begin{prop}\label{Asi}
We have $A_{\sigma(\la)}\simeq A(w_0\la)$ as graded algebras for any $\sigma\in W$, $\la\in X_-$.
\end{prop}
\begin{proof}
We first show that $A_{\sigma(\la)}$ is a quotient of $A(w_0\la)$. Let us consider the roots
(positive or negative) $\sigma(\al_1),\dots,\sigma(\al_n)$. For each $1\le i\le n$, let
$h^\sigma_i\in\fh$ be the Cartan element in the $\msl_2$-triple containing the root vector
corresponding to $\sigma(\al_i)$.
By~\cite[\S 6]{CFK} we know that there exists an isomorphism
from the ring of symmetric polynomials in $m_i$ variables $x^i_1,\dots,x^i_{m_i}$ to the space
${\rm U}(\C h^\sigma_i\T t\bC[t])v$, mapping $\sum_{j=1}^{m_i} (x^i_j)^k$ to $h^\sigma_i\T t^k$.
Clearly this map preserves degree.
Since $h^\sigma_i$, $1\le i\le n$ form a basis of $\fh$, we obtain a graded surjection
$A(w_0\la)\simeq B(w_0\la)\to A_{\sigma(\la)}$.

Now we show that $\ch A_{\si(\la)} \le \ch A(w_0\la)=(q)_{w_0\la}^{-1}$, which will imply that the above surjection is in fact an isomorphism. For this we use
the fusion product procedure with one factor, as follows. Let us start with $\W(w_0\la)$ and consider a generator $v_\sigma$ for the
$\sigma(\la)$-weight space in degree $0$. It is easy to see (because $\W(w_0\la)$ is a $\fg$-module) that
the character of ${\rm U}(\fh\T t\bC[t])v_\sigma$ is equal to $(q)_{w_0\la}^{-1}$.
Consider the filtration $F_s$ on $\W(w_0\la)$ given by $F_0=\bC v_\sigma$ and
\[
F_s={\rm span}\{(x_1\T (t-1)^{k_1}\dotsm x_l\T (t-1)^{k_l})v_\sigma : x_i\in\fg,\, k_1+\dots+k_l\le s\}.
\]
Then the associated graded $\fn^{af}$-module ${\rm gr}\,F_\bullet$ is a quotient of
$\W_{\sigma(\la)}$. In fact, $\lim_{s\to\infty} F_s=\W(w_0\la)$ and all the defining relations
of $\W_{\sigma(\la)}$ hold in ${\rm U}(\fn^{af})v_\sigma$. In other words,
$\widehat\si (f_{-\al}\T t)v_\sigma=0$ and $\widehat\si (e_{\al}\T 1)^{-\langle \al^\vee,\la\rangle +1}v_\sigma=0$
for all $\al\in\Delta_+$. We note that these relations hold even before passing to the associated graded module.
In fact, the operators we are interested in are either the powers of $e_\al\T 1$, or the powers of
$f_{-\al}\T t$, acting as $f_{-\al}\T (t-1)$. In the first case the required relations are obtained immediately
(since the degree zero subspace of $\W(w_0\la)$ is the irreducible $\fg$-module $V(w_0\la)$). In the second
case the power of $f_{-\al}\T (t-1)$ is the sum of several summands and each of the summands is the product of
factors of the form $f_{-\al}\T t$ or $f_{-\al}\T 1$. Now the desired relations follow from the statement that the
$\fh$-weights of $\W(w_0\la)$ and that of $V(w_0\la)$ coincide.

We thus obtain a surjection
$\W_{\sigma(\la)}\to {\rm U}(\fn^{af})v_\sigma$ and hence the surjection $A_{\sigma(\la)}\to {\rm U}(\fh\otimes t\C[t])v_\si$, which has character $(q)_{w_0\la}^{-1}$.
\end{proof}

\comment{On the other hand (see \cite[Remark 11]{FL}), we have ${\rm U}(\fn^{af})v_\si\simeq{\rm U}(\fn^{af})\bar{v}_\si$, where $\bar{v}_\si$ is the image of $v_\si$ in ${\rm gr}\,F_\bullet$. Hence we have a surjection $A_{\si(\la)}\to{\rm U}(\fh\otimes t\C[t])v_\si$.}

\begin{cor}\label{estimate}
The graded character of generalized global Weyl modules satisfy the estimate
\[
\ch \W_{\sigma(\la)}\le \frac{\ch W_{\sigma(\la)}}{(q)_{w_0\la}},
\]
where the inequality means the coefficient-wise inequality.
The inequality is an equality if and only if the $A_{\sigma(\la)}$ action on $W_{\sigma(\la)}$ is free.
\end{cor}

\begin{proof}Let $A_{\si(\la)}^+$ be the span of positive degree terms of $A_{\si(\la)}$.
Thanks to Lemma~\ref{bi-module} the generalized Weyl module $\W_{\si(\la)}$ is a bimodule over $\fn^{af}\times A_{\si(\la)}$.
By definition, the quotient $\W_{\si(\la)}/A_{\si(\la)}^+\W_{\si(\la)}$ is isomorphic to the local Weyl module
$W_{\si(\la)}$. Therefore,
\begin{align*}
\ch \W_{\si(\la)}\le \ch W_{\si(\la)}\ch A_{\si(\la)}
\end{align*}
with equality if and only if the action of $A_{\si(\la)}$ on $W_{\si(\la)}$ is free.
\end{proof}


Let us define the generalized global Weyl modules with characteristics.
Namely, let $\W_{\sigma(\la)}(m)$ be the cyclic $\fn^{af}$-module defined by the relations
\[
\widehat{\sigma}(f_{-\alpha}\T t)v=0,\quad \widehat{\sigma}(e_{\alpha})^{l_{\al,m}+1}v=0\quad \text{for all $\al\in \Delta_+$},
\]
where $l_{\al,m}= -{\langle \alpha^\vee, \lambda \rangle-|\{\beta^i_l : {\rm Re}\,\beta^i_l=-\alpha, l \leq m \}|}.$
As in the local case, the modules with characteristics do depend on the coroots $\beta^i_m$, but we omit them in the notation.
Let $A_{\sigma(\lambda)}(m)=\U(\fh\otimes t\C[t])v$ be the weight $\sigma(\la)$ subspace of $\W_{\sigma(\la)}(m)$
equipped with the natural structure of a graded algebra. We have the following lemma.
\begin{lem}
$A_{\sigma(\lambda)}(m)$ acts on $\W_{\sigma(\la)}(m)$ from the right making it into a bimodule over $\fn^{af}\times A_{\sigma(\lambda)}(m)$. For any $i$ such that $\langle \al_i^\vee,\la\rangle <0$ and any
$m>0$ the algebra $A_{\sigma(\lambda)}(m)$ is a quotient of $A_{\sigma(\lambda+\om_i)}$.
\end{lem}
\begin{proof}
The (right) action of $A_{\sigma(\lambda)}(m)$ is defined by the standard formula $(rv).a=r(av)$. Using
the $\msl_2$ theory (see also Lemma \ref{bi-module}), one checks that the action is well-defined.

To prove the second claim it suffices
to check that all the defining relations of $A_{\sigma(\la+\om_i)}\simeq A(w_0(\la+\om_i))$ hold in $A_{\sigma(\lambda)}(m)$
for positive $m$. We note that in type $A_1$ the claim follows from Proposition \ref{Asi} since for $\msl_2$ the
length of $t_{-\om_1}$ equals to one ($r=1$) and hence the module with characteristics coincide with the generalized
Weyl modules.
Now let us fix $h_j^\si$ as in the proof of Proposition \ref{Asi}. Then we need to show that for any $1\le j\le n$
the operators $h_j^\si\T t^k$, $k>0$ acting on the cyclic vector $v$ of $\W_{\si(\la)}(m)$ satisfy all the relation of the power
sums $\sum_{b=1}^{\langle\al_j^\vee,\la+\om_i\rangle} x_b^k$ (here $x_b$ are auxiliary variables).  For a fixed $j$ the element
$h_j^\si$ is the Cartan element in the $\msl_2$-triple corresponding to the weight $\si(\al_j)$. Assume first that $j\ne i$.
Then we have the relations $\widehat\si (f_{\al_j}\T t)v=0$ and $\widehat\si (e_{\al_j}\T 1)^{\langle\al_j^\vee,\la\rangle}v=0$ in
$\W_{\si(\la)}(m)$. Since for $j\ne i$ we have $\langle\al_j^\vee,\la\rangle=\langle\al_j^\vee,\la+\om_i\rangle$, the claim follows from the
type $A_1$ result. For $i=j$, since ${\rm Re}\,\beta_1^i=-\al_i^\vee$, we have the following relations in  $\W_{\si(\la)}(m)$
for all positive $m$:
\[
\widehat\si (f_{\al_i}\T t)v=0,\quad \widehat\si (e_{\al_i}\T 1)^{\langle\al_i^\vee,\la+\om_i\rangle}v=0,
\]
which imply the claim
\end{proof}

\begin{cor}
For any $m>0$  and $1\le i\le n$ such that $m_i>0$ one has
\[
\ch \W_{\sigma(\la)}(m)\le \frac{\ch W_{\sigma(\la)}(m)}{(q)_{w_0(\la+\om_i)}}.
\]
\end{cor}

\subsection{Structure theorems}
\label{SS:struct}
Let $\la=-\sum_{i=1}^n m_i\om_i\in X_-$.
\begin{lem}\label{m=1}
If $\sigma({\rm Re}\,\beta^i_1)\in\De_-$, then
\begin{equation}\label{neg}
\ch W_{\sigma(\la)}=(1-q^{m_i})\ch W_{\sigma(\la)}(1)+ \ch W_{\sigma s_{{\rm Re}\,\beta^i_1}(\la)}.
\end{equation}
If $\sigma({\rm Re}\,\beta^i_1)\in\De_+$, then
\begin{equation}\label{pos}
\ch W_{\sigma(\la)}=(1-q^{m_i})\ch W_{\sigma(\la)}(1)+ q^{m_i}\ch W_{\sigma s_{{\rm Re}\,\beta^i_1}(\la)}.
\end{equation}
\end{lem}
\begin{proof}
Let $s=s_{{\rm Re}\, \beta^i_1}$. Since ${\rm Re}\,\be^i_1=-\al_i^\vee$, there exist edges $\si\to\si s$ and $\si s\to \si$ in $\QBG$.
By Theorem~\ref{localmain}, it follows that there exist exact sequences of $\fn^{af}$-modules (with possibly a grading shift in the kernel)
\begin{gather}
0\to W_{\sigma s(\la)}(1)\to W_{\sigma(\la)}\to W_{\sigma(\la)}(1)\to 0,\\
0\to W_{\sigma(\la)}(1)\to W_{\sigma s(\la)}\to W_{\sigma s(\la)}(1)\to 0.
\end{gather}
Let us consider two cases. First, suppose that $\sigma({\rm Re}\,\beta^i_1)$ is negative (recall that ${\rm Re}\,\beta^i_1$ is always
negative). Then we have the following identities for the characters
\begin{align*}
\ch W_{\sigma(\la)} &= \ch W_{\sigma(\la)}(1) + \ch W_{\sigma s(\la)}(1),\\
\ch W_{\sigma s(\la)} &= \ch W_{\sigma s(\la)}(1) + q^{m_i}\ch W_{\sigma(\la)}(1).
\end{align*}
Substituting the expression for $\ch W_{\sigma s(\la)}(1)$ from the second equation to the first equation, we
obtain \eqref{neg}.

Now assume that $\sigma({\rm Re}\, \beta^i_1)$ is positive. Then one has
\begin{align*}
\ch W_{\sigma(\la)}&= \ch W_{\sigma(\la)}(1) + q^{m_i}\ch W_{\sigma s(\la)}(1),\\
\ch W_{\sigma s(\la)}&= \ch W_{\sigma s(\la)}(1) + \ch W_{\sigma(\la)}(1).
\end{align*}
The two equations imply \eqref{pos}.
\end{proof}

\begin{lem}\label{notiso}
Assume that $m>0$ and the surjection $$W_{\sigma(\la)}(m)\to W_{\sigma(\la)}(m+1)$$ is not an isomorphism. Then the surjection
$$W_{\sigma s_{{\rm Re}\, \beta^i_{m+1}}(\la)}(m)\to W_{\sigma s_{{\rm Re} \beta^i_{m+1}}(\la)}(m+1)$$ is an isomorphism.
\end{lem}
\begin{proof}
The map $W_{\sigma(\la)}(m)\to W_{\sigma(\la)}(m+1)$ is not an isomorphism if and only if there exists an edge
$\sigma\to \sigma s_{{\rm Re} \beta^i_{m+1}}$  in the quantum Bruhat graph. So our lemma can be rephrased in the following way:
if there is an edge $\sigma\to \sigma s_{{\rm Re} \beta^i_{m+1}}$ in QBG and $m>0$, then there is no edge backwards
$\sigma s_{{\rm Re} \beta^i_{m+1}}\to\sigma$. We note that two edges $\sigma\to \sigma s_\al$ and $\sigma s_\al\to \sigma$
are both present in the QBG if and only if $\langle \al^\vee,\rho\rangle=1$, which is true for simple coroots $\al^\vee$ only.
However, $\beta^i_{m+1}$ is simple if and only if $m=0$.
\end{proof}

Now let us consider the generalized global Weyl module with characteristics $\W_{\sigma(\la)}(m)$.
We note that $\W_{\sigma(\la)}(0)\simeq \W_{\sigma(\la)}$ and $\W_{\sigma(\la)}(r)\simeq \W_{\sigma(\la+\om_i)}$, where
$r=l(t_{-\om_i})$. We consider the surjection
\begin{equation}\label{map}
\W_{\sigma(\la)}(m)\to \W_{\sigma(\la)}(m+1)
\end{equation}
mapping cyclic vector to cyclic vector. The kernel of this map is given by
\begin{equation}\label{kernel}
{\rm U}(\fn^{af}) \widehat{\sigma}(e_{{\rm Re}\beta^i_{m+1}})^{l_{{\rm Re}\beta^i_{m+1},m}} v.
\end{equation}

\begin{lem}\label{L:to-ker}
Let $m_i>0$ and $s=s_{{\rm Re} \beta^i_{m+1}}$.
Assume that the map $\W_{\sigma(\la)}(m)\to \W_{\sigma(\la)}(m+1)$ is not an isomorphism.
Then there exists a surjective map of $\fn^{af}$-modules
\begin{equation}\label{isosurj}
\W_{\sigma s(\la)}(m) \to {\rm ker} (\W_{\sigma(\la)}(m)\to \W_{\sigma(\la)}(m+1))
\end{equation}
mapping the cyclic vector of $\W_{\sigma s(\la)}(m)$ to $(e_{{\rm Re} \beta^i_{m+1}})^{l_{{\rm Re} \beta^i_{m+1},m}} v$.
\end{lem}
\begin{proof}
Recall the local version of map \eqref{isosurj}: Theorem~\ref{localmain} gives a surjection (in fact, an isomorphism)
of the $\fn^{af}$-modules
$$W_{\sigma s(\la)}(m+1) \to {\rm ker} (W_{\sigma(\la)}(m)\to W_{\sigma(\la)}(m+1)).$$
If $m>0$, Lemma \ref{notiso} gives $W_{\sigma s(\la)}(m+1)\simeq W_{\sigma s(\la)}(m)$ and
hence $W_{\sigma s(\la)}(m)$ surjects onto the above mentioned kernel.
Now the proof for global Weyl modules proceeds exactly as in \cite[Theorem~2.18]{FM3}.

So we are left with the case $m=0$. We want to show that there exists a surjection
$\W_{\sigma s(\la)}\to{\rm ker} (\W_{\sigma(\la)}\to \W_{\sigma(\la)}(1))$. The proof goes along the
same lines as in the case $m>0$ with the only difference: we have to show that
\[
\widehat\sigma(f_{-\al_i}\T t)^{m_i+1} \widehat\sigma(e_{\al_i})^{m_i}v=0
\]
(instead of $\widehat\sigma(f_{-\al_i}\T t)^{m_i} \widehat\sigma(e_{\al_i})^{m_i}v=0$ in the local situation).
The equality is clear from the $A_1$ case.
\end{proof}

\begin{rem}
We show below that the map \eqref{isosurj} is in fact an isomorphism. In particular, for $m=0$ the kernel is isomorphic to
$\W_{\sigma s(\la)}\simeq {\rm U}(\fn^{af})v_\sigma\subset \W_{\sigma(\la)}$, where $v_\si\in \W_{\sigma(\la)}$ is the
$\fh$-weight $\si(\la)$ degree zero extremal vector.
\end{rem}

Now we are ready to prove the following structure theorem.

\begin{thm}\label{T:struct}
Let $\la\in X_-$.
\begin{enumerate}
\item If there is no edge $\sigma \to \sigma s_{{\rm Re} \beta_{m+1}}$ in $\QBG$, then the surjection \eqref{map} is an isomorphism.
If the edge does exist, then the kernel of \eqref{map} is isomorphic to the generalized global Weyl module with characteristics
$\W_{\sigma_1(\la)}(m)$ for $\sigma_1=\sigma s_{{\rm Re} \beta^i_{m+1}}$.\\
\item For any $\sigma\in W$ the module $\W_{\sigma(\la)}$ is free over $A_{\sigma(\la)}$ and the quotient is isomorphic to
the local Weyl module $W_{\sigma(\la)}$. In particular,
\[
\ch \W_{\sigma(\la)}=\frac{\ch W_{\sigma(\la)}}{(q)_{w_0\la}}.
\]
\item For any $m>0$ we have
\[
\ch \W_{\sigma(\la)}(m)=\frac{\ch W_{\sigma(\la)}(m)}{(q)_{w_0(\la+\om_i)}}.
\]
\end{enumerate}
\end{thm}
\begin{proof}
First, let $m=0$ and let us prove (1) and (2) by induction on $\ell(\sigma)$.
If $\ell(\sigma)=0$ then (2) is known (Theorem~\ref{loc-glob-classical}). Assume that (2) is known for $\sigma$ and suppose that $\si s>\si$ for $s=s_i, 1\le i \le n$. If $\langle \al_i^\vee, \la \rangle =0$, then $\si s_i(\lambda)=\lambda$ and we have $(2)$ for $\sigma s_i$. We may therefore assume that $\langle \al_i^\vee,\la\rangle <0$. Since $\si s>\si$, we have $\si(\al_i^\vee)>0$. Since ${\rm Re}\,\be_1^i=-\al_i^\vee$, Lemma~\ref{m=1} gives
\begin{equation}\label{ind}
\frac{\ch W_{\sigma(\la)}}{(q)_{w_0\la}}=\frac{\ch W_{\sigma(\la)}(1)}{(q)_{w_0(\la+\om_i)}}+
\frac{\ch W_{\sigma s(\la)}}{(q)_{w_0\la}}
\end{equation}
where $s=s_{{\rm Re} \beta^i_1}=s_i$. The left-hand side is the character of the global Weyl module $\W_{\sigma(\la)}$. The first summand on the right-hand side is greater than or equal to $\ch \W_{\sigma(\la)}(1)$ and the second summand satisfies
\[
\frac{\ch W_{\sigma s(\la)}}{(q)_{w_0\la}}\ge \ch \W_{\sigma s(\la)}\ge \ch \eqref{kernel}.
\]
Since $\ch \W_{\si(\la)} = \ch \W_{\si(\la)}(1)+\ch \eqref{kernel}$ by Lemma~\ref{L:to-ker}, formula \eqref{ind} implies that
\begin{itemize}
\item $\ch \W_{\sigma(\la)}(1)=\ch W_{\sigma(\la)}(1)/(q)_{w_0(\la+\om_i)}$,
\item $\eqref{kernel}\simeq \W_{\sigma s(\la)}$,
\item $\ch \W_{\sigma s(\la)}=\ch W_{\sigma s(\la)}/(q)_{w_0\la}$, i.e., $A_{\sigma s(\la)}$ acts freely.
\end{itemize}
Thus we have proved (1) for $\sigma$ and (2) for $\sigma s_i$. This gives the induction step.

Now let us go further with $m=2,3,\dots$ in the chain of maps
\[
\W_{\sigma(\la)}\to \W_{\sigma(\la)}(1)\to \W_{\sigma(\la)}(2)\to \dots.
\]
Assume that for some $m\ge 1$ we know that $\ch \W_{\sigma(\la)}(m)=\ch W_{\sigma(\la)}(m)/(q)_{w_0(\la+\om_i)}$ and the surjection
$\W_{\sigma(\la)}(m)\to \W_{\sigma(\la)}(m+1)$ is not an isomorphism. According to the local statement (Theorem~\ref{localmain}), we have
\[
\ch W_{\sigma(\la)}(m)=\ch W_{\sigma(\la)}(m+1) + \ch W_{\sigma s_{{\rm Re} \beta^i_m(\la)}}(m+1).
\]
Using Lemma \ref{notiso} we replace $\ch W_{\si s_{{\rm Re} \beta^i_m(\la)}}(m+1)$ by
$\ch W_{\si s_{{\rm Re} \beta^i_m(\la)}}(m)$ and obtain
\[
\ch \W_{\sigma(\la)}(m)=\ch W_{\sigma(\la)}(m+1)/(q)_{w_0(\la+\om_i)}+\ch W_{\sigma s_{{\rm Re} \beta^i_m(\la)}}(m)/(q)_{w_0(\la+\om_i)}.
\]
Since the first summand on the right-hand side is greater than or equal to the character of $\W_{\sigma(\la)}(m+1)$ and the second summand is greater than or equal to the character of the kernel of the map $\W_{\sigma(\la)}(m)\to \W_{\sigma(\la)}(m+1)$, we argue as we did in (2) to prove (3) for $m+1$. This completes the proof.
\end{proof}

\begin{cor}
The surjection \eqref{isosurj} is an isomorphism.
\end{cor}

As another corollary, we obtain the following theorem.
\begin{thm}\label{filtration}
Fix $\sigma\in W$ and $\la\in X_-$ such that $-m_i=\langle\al_i^\vee,\la\rangle<0$ for some $i=1,\dots,n$.
Then we have an embedding of $\fn^{af}\oplus \fh \T 1\oplus \bC d$-modules
\[
\W^{\sigma(\la),m_i}_{\sigma(\la)}\to \W^{\si(\la)}_{\sigma(\la)}
\]
with the cyclic vector $v$ mapped to $\hat\sigma (f_i\T t)^{m_i} (\hat\sigma e_i)^{m_i}v$.
The cokernel can be filtered in such a way that each subquotient is isomorphic to a generalized global Weyl module of the form
$\W^\nu_{\tau(\la+\om_i)}$ for some $\tau\in W$ and some affine weight $\nu$.
The number of these generalized global Weyl modules is equal to the dimension of the fundamental
local Weyl module $W(\om_i)$ and the labeling set consists of the paths in QBG of the form
\[
\si\to \si s_{{\rm Re}\be^i_{m_1+1}}\to \si s_{{\rm Re}\be^i_{m_1+1}}s_{{\rm Re}\be^i_{m_2+1}}\to\dots
\]
for a sequence $0\le m_1<m_2<\dots <r$.
For a path $p$ the corresponding element $\tau$ is the end of $p$.
\end{thm}


\comment{
\begin{rem}
We have monomial bases for the global Weyl modules compatible with the filtration described above. Each vector in the basis
is a monomial in variables $e_\al\T 1$ and $f_\al\T t$ applied to the cyclic vector.
\end{rem}
}

\section{Nonsymmetric $q$-Whittaker functions}
\label{S:whitt}

In this section we express the characters of generalized local and global Weyl modules in terms of the $E_\la^\si$ and establish the connection to nonsymmetric $q$-Whittaker functions.

\subsection{Scalar product}
The nonsymmetric Macdonald polynomials $E_\la(X;q,v)$ form an orthogonal basis of $\K[X]$ with respect to Cherednik's scalar product from \cite{Ch-ann}. Their squared norms are given explicitly by (see \cite[(7.5)]{M2} and \cite[(5.5)]{Ch1}):
\begin{align}\label{E:norm-E}
\langle E_\la, E_\la \rangle &= \prod_{\al+m\de^Y\in\Inv(m_\lambda)}\frac{(1-q^mt^{-\pair{\rho^{\vee Y}}{\alpha}-1})(1-q^mt^{-\pair{\rho^{\vee Y}}{\alpha}+1})}{(1-q^mt^{-\pair{\rho^{\vee Y}}{\alpha}})^2}.
\end{align}
We refer to \cite{Ch1} and \cite[Chapter 5]{M3} for proofs of \eqref{E:norm-E}. We note the following property of the minimal representatives $m_\la$ (see \cite[(2.4.7)]{M3}):
\begin{align}\label{E:inv-m}
\al+m\de^Y\in\Inv(m_\la)\Rightarrow\al\in \De_-(Y).
\end{align}

For our purposes, \eqref{E:norm-E} will be sufficient as the definition of the scalar product, which is linear in the first argument and $*$-linear in the second argument (where $q^*=q^{-1}$ and $v^*=v^{-1}$, naturally extended to fractional powers as needed). The adjoint of $T_\si$ with respect to this scalar product is $T_\si^{-1}$ (see \cite[(5.1.22)]{M3}), and hence for any fixed $\si\in W$, the basis $\{E_\la^\si\}_{\la\in X}$ of $\K[X]$ is orthogonal and $\langle E_\la^\si, E_\la^\si\rangle = \langle E_\la, E_\la \rangle.$

With \eqref{E:inv-m} in hand, the norm formula \eqref{E:norm-E} immediately gives
\begin{align*}
\langle E_\la, E_\la \rangle_{v=0} = \prod_{\substack{\al+m\de^Y\in\Inv(m_\lambda)\\\text{$-\alpha$ simple}}}(1-q^m)
\end{align*}
and hence
\begin{align}
\label{E:norm-E-0-anti}
\langle E_\lambda, E_\lambda \rangle_{v=0} =& \prod_{i\in I}\prod_{j=1}^{-(\la,\al_i^Y)}(1-q^j),\quad \text{for $\lambda\in X_-$.}
\end{align}
To get \eqref{E:norm-E-0-anti} we use the fact that, for $\la\in X_-$, we have $m_\la = t_\la$ and
\begin{align*}
\Inv(t_\la)=\{\al+m\de^Y\in\De(\tY) : 0<(\la,\al)\leq m\}.
\end{align*}
In the untwisted case, we have $\al_i^Y = \al_i^{\vee X}$ and $(\la,\al_i^Y)=\pair{\al_i^{\vee X}}{\la}$ for all $i\in I$.

\subsection{Characters of generalized Weyl modules}\label{SS:weyl-mac}
Suppose that $\De(\tX)$ is untwisted. Let $\la\in X_-$ and $\si\in W$ be arbitrary. Comparing \eqref{E:RY-limit-0} with \cite[Theorem B]{FM3}, we arrive at
\begin{align}
\label{E:mac-weyl}
\ch W_{\si(\la)} &= E_\la^\si(X;q,0).
\end{align}
Now comparing Theorem~\ref{T:struct}(2) and \eqref{E:norm-E-0-anti}, we deduce that
\begin{align}
\label{E:mac-glob}
\ch \W_{\si(\la)} &= \frac{E_\la^\si(X;q,0)}{\langle E_\la,E_\la\rangle_{v=0}}.
\end{align}

\subsection{Assumption}
The nonsymmetric $q$-Whittaker function is defined only when $\De(\tX)$ is dual untwisted (see \S\ref{SS:af}). Thus we make the following assumption for the remainder of the paper: $\De(\tX)$ is dual untwisted.

Therefore we have $\De(X)=\De(Y)$, $\De(\tX)=\De(\tY)$, and the pairing $(\cdot,\cdot) : X\times Y=X\times X\to \Q$ is normalized so that $(\al,\al)=2$ for all $\al\in\De_s(X)$. Recall that $e$ is the smallest positive integer such that $(X,Y)\subset\frac{1}{e}\Z$. When there is no possibility of confusion, we will drop superscripts $X$ and $Y$, e.g., we will write $\rho=\rho^X=\rho^Y$. We will, however, distinguish notationally between $X$ and $Y$ when it adds conceptual clarity.

The connection to characters of generalized Weyl modules via \eqref{E:mac-weyl} and \eqref{E:mac-glob} is available when $\De(\tX)$ is untwisted. The overlap with our standing assumption consists of exactly the simply-laced $\De(\tX)$. (Generalized Weyl modules for twisted affinizations are investigated in \cite{FM4}.)

\subsection{Nonsymmetric basic hypergeometric function}
The nonsymmetric $q$-Whittaker function is defined as a limit of the nonsymmetric basic hypergeometric function from \cite{Ch-mehta}. We will follow the notation of \cite[\S3.1]{CO} to introduce the latter, but we also refer the reader to \cite{Sto-dft} and \cite[\S 2.5]{Sto-hc} for more details.

Recall that $\LL=\Q(q^{\frac{1}{2e}},v)$, the extension of $\K$ required for the definition of DAHA and the polynomial module.
Consider the formal series
\begin{align}
\gamma(X)=\sum_{\lambda\in X}q^{\frac{(\lambda,\lambda)}{2}} X^\lambda.
\end{align}
The series $\ga(X)$ can be regarded as an element of $\LL[[X]] = \prod_{\la\in X} \LL X^\la$. See Remark~\ref{R:functions} for another interpretation of $\ga(X)$.

The nonsymmetric basic hypergeometric function \cite{Ch-mehta} is
\begin{align}
G(Z,X) &= (\gamma(Z)\gamma(X))^{-1}\sum_{\lambda\in X} q^{\frac{(\lambda,\lambda)}{2}}v^{-(\lambda_-,2\rho)}\,\frac{E_\lambda(Z)^* E_\lambda(X)}{\langle E_\lambda, E_\lambda\rangle},
\end{align}
where we write $E_\la(X)=E_\la(X;q,v)$ for short and $E_\la(Z)\in\K[Z]=\K[Z^\la : \la\in X]$ are the same nonsymmetric Macdonald polynomials but in a different set of variables $Z^\la$ for $\la\in X$. We define $(Z^\mu)^*=Z^{-\mu}$, as we did in $\K[X]$.

It is often more convenient to work with the auxiliary series
\begin{align}
\Xi(Z,X) &= \gamma(Z)\gamma(X)G(Z,X) = \sum_{\lambda\in X} q^{\frac{(\lambda,\lambda)}{2}}v^{-(\lambda_-,2\rho)}\,\frac{E_\lambda(Z)^* E_\lambda(X)}{\langle E_\lambda, E_\lambda\rangle}
\end{align}
We will refer to $\Xi(Z,X)$ the {\em series part} of $G(Z,X)$. These definitions of $G$ and $\Xi$ differ from \cite{Ch-mehta} by a multiplicative constant, which we do not need in this paper.

\comment{
This is well-defined as a series in $(\Q[t][X,Z])[[q^{1/2}]]$, due to the exponent $(\la,\la)/2$ of $q$. Problem: Translation operators are not defined here? Therefore, it makes sense to consider the specialization $\Xi(Z,X)\big|_{t=0}$.
}

\begin{rem}\label{R:functions}
There are different ways to understand these ``functions.'' Probably what is conceptually the simplest is to assume that $q^{\frac{1}{2e}},v$ are complex numbers with $0<|q^{\frac{1}{2e}}|,|v|<1$. Then $G$ and $\Xi$ can be regarded as meromorphic functions on $(X\otimes_\Z \C)^2$, where $Z^\mu X^\nu$ is the function sending $(z,x)\in(X\otimes_\Z\C)^2$ to $q^{(\mu,z)+(\nu,x)}$. In this setting, all of the infinite sums above are uniformly absolutely convergent on compact sets---see \cite[Lemma~9.2]{Sto-dft} for the relevant estimates. 

If one wishes to work purely algebraically, then one may regard $G$ and $\Xi$ as elements of the subspace of $(\Q[v]((q^{\frac{1}{2e}})))[[X,Z]]=\prod_{\la,\mu\in X} \Q[v]((q^{\frac{1}{2e}}))Z^\la X^\mu$ consisting of $(a_{\la\mu}X^\la Z^\mu)_{\la,\mu\in X}$ such that the order of $a_{\la\mu}$ with respect to $q^{1/(2e)}$ grows at least as fast as $(\la,\la)+(\mu,\mu)$.\footnote{For instance, it would be sufficient to require that there exists a constant $C>0$, depending only on the element $(a_{\la\mu}X^\la Z^\mu)$, such that $\ord(a_{\la\mu})\geq C((\la,\la)+(\mu,\mu))$ whenever $(\la,\la)+(\mu,\mu)$ is sufficiently large.} Multiplication by $\gamma(X)$ (or $\gamma(Z)$) is a well-defined invertible endomorphism of this space.

In either of these two settings, the operators from the polynomial module of DAHA acting on either the $Z$ or $X$ variables (but not both at the same time) extend to these larger spaces of functions.\qed
\end{rem}

The function $G(Z,X)$ has remarkable symmetry (see \cite[Theorem 5.4]{Ch-mehta}):
\begin{align}\label{E:G-duality}
H_Z\cdot G(Z,X) = \varphi(H)_X \cdot G(Z,X),\quad \text{for all $H\in\HH$},
\end{align}
expressed via the duality anti-involution $\varphi:\HH\to\HH$ of \cite{Ch-kz} (see also \cite{I-inv}):
\begin{align}\label{E:phi-T}
\varphi(T_i)&=T_i\\
\label{E:phi-X}
\varphi(X^\la)&=Y^{-\la}
\end{align}
for $i\in I$ and $\la\in X$.
Since $\varphi$ is an involution, the above formulas uniquely determine it. We use subscripts in \eqref{E:G-duality} to indicate the variables upon which the operators are acting. In either set of variables, the operators are those from the polynomial module of DAHA.
\begin{rem}
The assumption that $\De(\tX)$ is dual untwisted is essential here, as the function $G(Z,X)$ exists only in this case. When $X\neq Y$, there does exist an analogue of the anti-homomorphism $\varphi$, but it maps between two DAHAs of different types (see \cite[Theorem~5.11]{H} and \cite[Theorem~4.7]{Sto-surv}).
\end{rem}

\subsection{Nonsymmetric $q$-Whittaker function}
\label{SS:whitt}
Simultaneously extending results of \cite{CM2} (in the $\fsl_2$-case) and \cite{Ch-whitt} (in the symmetric case), the nonsymmetric $q$-Whittaker function $\Omega(Z,X)$ was defined in \cite{CO} by applying a limiting procedure to $G(Z,X)$ in the variable $Z$. In order to make our connection to $\Omega(Z,X)$, we now recall the definition of this limiting procedure.

Let $\cF$ be a suitable $\LL$-vector space of ``functions'' containing $G(Z,X)$ and carrying an action of the the difference-reflection operators from the polynomial module of DAHA; see Remark~\ref{R:functions} for two possibilities for $\cF$.

\comment{
Consider the algebra $\hcF=\Fun(W,\cF)$ of all functions $f:W\to\cF$ under pointwise multiplication. Let $W$ act on $\hcF$ by
$
\si\cdot f(\tau) = f(\si^{-1}\tau).
$
That is, we ignore the $W$-action on $\cF$. Let $e_\si : W\to\cF$ be the characteristic function: $e_\si(\tau)=\de_{\si\tau}$, where $\de_{\si\tau}$ is the Kronecker delta and $1$ means the unit element of $\cF$.
}

Consider the group algebra $\LL[W]=\oplus_{\si\in W}\LL e_\si$, where $e_\si$ denotes the standard basis vector indexed by $\si$. Form the space\footnote{More precisely, if one wishes to equip $\hcF$ with an algebra structure, one should take $\hcF=\Fun(W,\cF)$, the space of all functions on $W$ with values in $\cF$ (under pointwise multiplication). We will not need to make this technical distinction here.} $\hcF=\cF\otimes\LL[W]$. An arbitrary element of $\hcF$ can be represented as a linear combination $\sum_{\si\in W} f_\si e_\si$ with $f_\si\in\cF$, where we omit the tensor product for simplicity of notation. We call $f_\si$ the $\si$-component of $f$.

We now impose the relation $v=q^{k/2}$ for $k\in\Z$.
The nonsymmetric $q$-Whittaker function \cite[\S3.4]{CO} can be defined as follows:
\begin{align}
\label{E:NSWhitt-def-Xi}
\Omega(Z,X) &= (\gamma(Z)\gamma(X))^{-1}\Big(\sum_{\si\in W} \Gamma_{-k\rho}^Z \si_Z^{-1}\cdot \Xi(Z,X)\,e_\si\Big)\Big|_{v=0}.
\end{align}
Note that $\Gamma_{-k\rho}^Z\cdot Z^\la=v^{-(2\rho,\la)}Z^\la$, so the relation $v=q^{k/2}$ is more a notational convenience than a necessity (ultimately, after applying the operator $\Gamma_{-k\rho}^Z$, we can we work entirely in terms of $v$ and forget $k$).
As above, we call $\ga(Z)\ga(X)\Om(Z,X)$ the {\em series part} of $\Omega(Z,X)$. We note that in the analytic setting ($0<|q|<1$), the specialization $v=0$ corresponds to the limit as $k\to\infty$.

\begin{rem}\label{R:q-toda}
It is shown in \cite[Proposition 3.3]{CO} that $\Omega(Z,X)$ is well-defined. There is a compatible limiting procedure on the operators from the DAHA polynomial module $\K[Z]$, giving rise to an action of the nil-DAHA on $\hcF$ at $v=0$. See \cite[Theorem~3.4]{CO}, which also gives the analogue of \eqref{E:G-duality} for the nonsymmetric $q$-Whittaker function $\Om(Z,X)$. In particular, the action of the subalgebra $\K[Y]\subset\HH$ becomes in the limit a homomorphism $\Q(q^{1/(2e)})[Y]\to \End(\hcF_{v=0}), Y^\mu\mapsto\hY^\mu$, which is a nonsymmetric variant of the $q$-Toda system of difference operators. The operators $\hY^\mu$ are called $q$-Toda Dunkl operators. The analogue of \eqref{E:G-duality} in this special case is
\begin{align*}
\hY^\mu_Z\cdot\Omega(Z,X)&=X^{-\mu}\Omega(Z,X).
\end{align*}
\end{rem}

The following theorem improves on \cite[Proposition 3.3]{CO} by providing an explicit expression for the nonsymmetric $q$-Whittaker function. This new expression is manifestly $q$-positive and carries representation-theoretic meaning: in simply-laced types, one may use \eqref{E:mac-glob} to state the following theorem in terms of graded characters of generalized global Weyl modules.

\begin{thm}\label{T:whitt}
We have $\Omega(Z,X)=(\ga(Z)\ga(X))^{-1}\sum_{\si\in W}\Omega(Z,X)_\si\,e_\si$ where
\begin{align}
\Omega(Z,X)_\si = \sum_{\la\in X_-} q^{\frac{(\la,\la)}{2}} Z^{-\la} \frac{E_\la^\si(X;q,0)}{\langle E_\la,E_\la\rangle_{v=0}}
\end{align}
\end{thm}

\begin{proof}
We proceed by directly evaluating \eqref{E:NSWhitt-def-Xi}. To this end, we fix $\si\in W$ and consider the $\si$-component of \eqref{E:NSWhitt-def-Xi}. By \eqref{E:G-duality} and \eqref{E:phi-T}, we have
\begin{align}
G(Z,X) = (T_{\si^{-1}}^{-1})_Z\cdot(T_\si)_X\cdot G(Z,X).
\end{align}
Hence
\begin{align*}
&G(Z,X) = (\gamma(Z)\gamma(X))^{-1}\sum_{\lambda\in X} q^{(\lambda,\lambda)/2}v^{-(2\rho,\lambda_-)}\,\frac{(T_{\si^{-1}}^{-1}\cdot E_\lambda(Z)^*)(T_\si\cdot E_\lambda(X))}{\langle E_\lambda, E_\lambda\rangle}.
\end{align*}
Now
\begin{align*}
&(T_{\si^{-1}}^{-1}\cdot E_\lambda(Z)^*)(T_\si\cdot E_\lambda(X))\\
&\qquad=(v^{2\ell(\si\si_\lambda^{-1})-\ell(\si_\lambda^{-1})}T_{\si^{-1}}^{-1}\cdot E_\lambda(Z)^*)(v^{-2\ell(\si\si_\lambda^{-1})+\ell(\si_\lambda^{-1})} T_\si\cdot E_\lambda(X))\\
&\qquad=(v^{-2\ell(\si\si_\lambda^{-1})+\ell(\si_\lambda^{-1})}T_\si\cdot E_\lambda(Z))^*(v^{-2\ell(\si\si_\lambda^{-1})+\ell(\si_\lambda^{-1})} T_\si\cdot E_\lambda(X))\\
&\qquad=E_\la^\si(Z;q,v)^*E_\la^\si(X;q,v)
\end{align*}
where we use \eqref{E:*-T} in the second equality. Therefore, it suffices to show that
\begin{align}\label{E:Z-lim}
\lim_{v\to 0}v^{-(2\rho,\la_-)}\Gamma_{-k\rho}^Z \si_Z^{-1}(E_\la^\si(Z;q,v)^*) = \begin{cases}Z^{-\la_-}&\text{if $\la\in X_-$}\\0&\text{otherwise.}\end{cases}
\end{align}

By Corollary~\ref{C:TE-supp}, for any $\mu$ such that $[Z^{-\mu}]E_\la^\si(Z;q,v)^*\neq 0$, we have $\si^{-1}(\mu)\in\la_-+Q_+$, and $[Z^{-\si(\la_-)}]E_\la^\si(Z;q,v)^*=\de_{\la,\la_-}$. Hence $[Z^{-\si^{-1}(\mu)}]\si^{-1}_Z E_\la^\si(Z;q,v)\neq 0$ implies $\si^{-1}(\mu)\in\la_-+Q_+$, and $[Z^{-\la_-}]\si_Z^{-1}E_\la^\si(Z;q,v)=\de_{\la,\la_-}$. For any $\mu\in X$, let $a_{\la\mu}^\si(q,v)=[Z^{\si^{-1}(\mu)}]\si_Z^{-1}E_\la^\si(Z;q,v)$ and $b_{\la\mu}^\si(q,v)=[Z^{-u^{-1}(\mu)}]u_Z^{-1}E_\la^\si(Z;q,v)^*$. By the definition of $*$, we have $a_{\la\mu}^\si(q,v)=b_{\la\mu}^\si(q^{-1},v^{-1})$. By Proposition~\ref{P:RY-limit-inf}, the limit $\lim_{v\to\infty}a_{\la\mu}^\si(q,v)$ is well-defined. Hence $\lim_{v\to 0}b_{\la\mu}^\si(q,v)$ is well-defined. Now for any $\mu\in X$ such that $\si^{-1}(\mu)\in\la_-+Q_+$, we have
\begin{align*}
\lim_{v\to 0} v^{-(\la_-,2\rho)}\Gamma_{-k\rho}^Z \si_Z^{-1}(b_{\la\mu}^\si(q,v) Z^{-\mu})&= \lim_{v\to 0} v^{(\si^{-1}(\mu)-\la_-,2\rho)}b_{\la\mu}^\si(q,v) Z^{-\si^{-1}(\mu)}\\
&=\begin{cases}Z^{-\la_-}&\text{if $\si^{-1}(\mu)=\la_-=\la$}\\0&\text{otherwise.}\end{cases}
\end{align*}
This proves \eqref{E:Z-lim} and we are done.
\end{proof}

\section{The $\fsl_2$-case}
\label{S:A1}

In this section we work out the main constructions of the paper in the case $\fg=\msl_2$. We write $\om=\om_1$ for the single fundamental weight and $s=s_1$ for the corresponding reflection.

\subsection{Weyl modules}
The generalized global Weyl modules for $\fg=\msl_2$ are of the form $\W_{m\om}$, $m\in\bZ$.
They are defined by the relations
\begin{align*}
e^{-m+1}v&=0,\quad  (f\T t\bC[t])v=0,\quad \text{ for  $m\le 0$},
\end{align*}
and
\begin{align*}
(f\T t)^{m+1}v&=0,\quad e\T \bC[t]v=0,\quad \text{ for $m>0$}.
\end{align*}
If $m\ge 0$, then $\W_{-m\om}$
is isomorphic to the classical global Weyl module $W(m\om)$.
It is known that the algebra $A(m\om)$ acts freely on $\W(m\om)$ and the quotient
is isomorphic to the local Weyl module. In particular, one has
\begin{equation}\label{char}
\ch \W_{-m\om}=\frac{\ch W_{-m\om}}{(q)_m}
\end{equation}
and $(q)_m^{-1}$ is the character of the highest weight space $A(m\om)=A_{-m\om}$.

Let us show that $A_{m\om}\simeq A(m\om)$ for any $m\ge 0$. Indeed, the algebra $\fn^{af}$ is equipped with
an automorphism $I$, such that $I(e\T 1)=f\T t$, $I(e\T 1)=f\T t$, $I(h\T t^k)=-h\T t^k$.
It is easy to see that the $I$-twist of $\W_{m\om}$ is isomorphic to $\W_{-m\om}$.
Hence $A_{-m\om}\simeq A_{m\om}$ and this algebra acts freely on $\W_{m\om}$.

Now let us consider the decomposition procedure of $W_{-m\om}$, $m\ge 0$. We consider three vectors
$v, e^mv, (f\T t)^m e^mv$ in $\W_{-m\om}$, $m\ge 0$. It is easy to see that the following relations hold:
\begin{align*}
(f\T t)^{m+1}(e^mv)&=0,\quad e\T\bC[t](e^m v)=0,\\ (f\T t\bC[t])\left((f\T t)^m e^mv \right)&=0,\quad e^{m+1}\left((f\T t)^m e^mv \right)=0.
\end{align*}
So we have exact sequences for $m\ge 0$:
\begin{gather*}
\W_{m\om}\to {\rm U}(\fn^{af})e^mv\subset \W_{-m\om}\to \W_{(-m+1)\om}\to 0,\\
\W_{-m\om}\to {\rm U}(\fn^{af})(f\T t)^me^mv\subset \W_{m\om}\to \W_{(m-1)\om}\to 0.
\end{gather*}
In order to prove that the leftmost maps are in fact isomorphisms, we compare the characters. In particular, we need to prove that
\begin{equation}\label{dec-m}
\ch \W_{-m\om} = X^{-1} \ch \W_{(-m+1)\om} + \ch \W_{m\om},
\end{equation}
which is implied by
\[
\ch \W_{-m\om} = X^{-1} \ch \W_{(-m+1)\om} + X\ch \W_{(m-1)\om} + q^m\ch \W_{-m\om}.
\]
Using \eqref{char}, this equality is equivalent to
\[
\ch W_{-m\om} = X^{-1} \ch W_{(-m+1)\om} + X\ch W_{(m-1)\om},
\]
which is well known (see \cite{CP}).

One can perform the decomposition procedure for $\W_{m\om}$, $m\ge 0$ and derive the following analogue of \eqref{dec-m}:
\begin{equation}\label{dec+m}
\ch \W_{m\om} = X\ch \W_{(m-1)\om} + q^m\ch \W_{-m\om}.
\end{equation}

\subsection{Nonsymmetric $q$-Whittaker function}\label{NWf}
Here we follow \cite[\S2]{CM2}, but with some changes of variables. In particular, we change $X$ to $Z$ and $\La$ to $X$. We use the vector notation
\begin{align*}
\twovec{f_\id}{f_s} = f_\id e_\id + f_s e_s
\end{align*}
for elements of $\hcF$, and we write $Z=Z^\om$. The nonsymmetric $q$-Whittaker function (where we use $q^{x^2}$ as a substitute for $\gamma(X)$---see below):
\begin{align}\label{unit}
\Omega(Z,X) = q^{x^2}q^{z^2}\Bigg(1+\sum_{m>0}q^{m^2/4}\Bigg(\frac{E_{-m\om}(X;q,0)}{\prod_{i=1}^m(1-q^i)}\twovec{Z^m}{q^m Z^m}&\\
+\frac{E_{m\om}(X;q,0)}{\prod_{i=1}^{m-1}(1-q^i)}&\twovec{0}{Z^m}\Bigg)\Bigg)\notag
\end{align}
satisfies the eigenfunction equations
\begin{align}\label{E:A1-eigen'}
\hY_Z^{-1}\cdot\Omega(Z,X) &= X\cdot\Omega(Z,X)\\
\label{E:A1-eigen}
\hY_Z\cdot\Omega(Z,X) &= X^{-1}\cdot\Omega(Z,X)
\end{align}
where $\hY_Z^{-1}$ and $\hY_Z$ are the $q$-Toda Dunkl operators
\begin{align}\label{E:A1-toda-dunkl'}
\hY_Z^{-1}\cdot\twovec{f_\id}{f_s}&=\twovec{(1-Z^{-2})\Ga(f_\id)+\Ga^{-1}(f_s)}{\Ga^{-1}(f_s)-Z^{-2}\Ga(f_\id)}\\
\label{E:A1-toda-dunkl}
\hY_Z\cdot\twovec{f_\id}{f_s}&=\twovec{\Ga^{-1}(f_\id-f_s)}{\Ga(f_s)-\Ga(\frac{f_s-f_\id}{Z^2})}
\end{align}
acting on the variable $Z$, where $\Ga=\Ga_\om$ and hence $\Ga(Z^m)=q^{m/2}Z^m$. We regard $q^{z^2}$ as a formal function satisfying $\Ga(q^{z^2})=q^{1/4}Zq^{z^2}$ and $\Ga^{-1}(q^{z^2})=q^{1/4}Z^{-1}q^{z^2}$. (These are the same difference equations satisfied by $\gamma(Z)^{-1}$, which is the essential property we need.) Our notation is consistent with the logarithmic notation $Z=q^z$ and with $\Ga$ acting by the shift $z\to z+\frac{1}{2}$.

\subsection{Weyl and Whittaker}
\label{SS:A1-whitt}
In the case $\fg=\fsl_2$ we can see the connection between generalized Weyl modules and nonsymmetric $q$-Whittaker functions directly.

\begin{prop}\label{Omega}
\begin{align}
\Omega(Z,X) &= q^{x^2}q^{z^2}\sum_{m\geq 0} q^{m^2/4}Z^m\twovec{\ch \W_{-m\om}}{\ch \W_{m\om}}
\end{align}
\end{prop}

\begin{proof}
This is a consequence of the identities ($m\ge 0$)
\begin{align}
(q)_m\ch \W_{-m\om} &= E_{-m\om}(X;q,0)\\
(q)_m \ch \W_{m\om} &= E_{-m\om}(X^{-1};q^{-1},\infty)\\
&= \lim_{v\to 0} v^{-1}T_s\cdot E_{-m\om}(X;q,v)\notag\\
&= q^m E_{-m\om}(X;q,0)+(1-q^m)E_{m\om}(X;q,0).\notag\qedhere
\end{align}
\end{proof}

The eigenfunction equations \eqref{E:A1-eigen'} and \eqref{E:A1-eigen} translate to the following relations among the characters of generalized Weyl modules. These are proved by expanding the action of the operators $\hY^{-1}_Z$ and $\hY_Z$ on $\Om(Z,X)$ and equating the coefficients of $Z^m$ on both sides.

\begin{cor}\label{rr}
The graded characters of generalized global Weyl modules satisfy the following recurrence relations
\begin{align}
X\ch \W_{-m\om} &= \ch \W_{(-m+1)\om} - q^{m+1}\ch \W_{-(m+1)\om} +\ch \W_{(m+1)\om}\label{1}\\
X\ch \W_{m\om} &= \ch \W_{(m+1)\om} - q^{m+1}\ch \W_{-(m+1)\om}\label{2}\\
X^{-1}\ch \W_{-m\om} &= \ch \W_{-(m+1)\om} - \ch \W_{(m+1)\om}\label{3}\\
X^{-1}\ch \W_{m\om} &= \ch \W_{(m-1)\om} - q^m (\ch \W_{(m+1)\om}- \ch \W_{-(m+1)\om}).\label{4}
\end{align}
\end{cor}

\begin{rem}
We note that \eqref{2} and \eqref{3} imply \eqref{1} and \eqref{4}. Moreover, \eqref{2} is equivalent to \eqref{dec+m}
and \eqref{3} is equivalent to \eqref{dec-m}.
\end{rem}

\end{document}